\documentclass[]{amsart}       %Free format. Use for Archiv. Erase lines: 9,60,873;13,26-31,47-57
\long\def\dropnow#1{\relax}

\usepackage[utf8x]{inputenc}
\usepackage{ucs}
\usepackage{amssymb,amsmath,amsfonts,amsthm,latexsym,graphics,graphicx,xr}
\newtheorem{theorem}{Theorem}%[section]
\newtheorem{proposition}{Proposition}%[section]
\newtheorem{corollary}{Corollary}%[section]
\newtheorem{remark}{Remark}%[section]
\title{Symmetries of sub-Riemannian surfaces} %%%%%%%%%% INCLUDE FOR ARCHIVE %%%%%%%%%%%%


\date{August 26, 2009}
\author{Jos\'e Ricardo Arteaga \& Mikhail Malakhaltsev}%% INCLUDE FOR ARCHIVE %%%%%%%%%%%%
\address{Universidad de Los Andes, Bogota, Colombia}%% INCLUDE FOR ARCHIVE %%%%%%%%%%%%
\email{jarteaga@uniandes.edu.co}%% INCLUDE FOR ARCHIVE %%%%%%%%%%%%
\address{Kazan State University, Kazan, Russia}%% INCLUDE FOR ARCHIVE %%%%%%%%%%%%
\email{mikhail.malakhaltsev@ksu.ru}%% INCLUDE FOR ARCHIVE %%%%%%%%%%%%

\begin{document}
\maketitle %%%%%%%%%% INCLUDE FOR ARCHIVE %%%%%%%%%%%%
%------------------------------------
%\begin{frontmatter}   %%%%%%%%%% ERASE FOR ARCHIVE %%%%%%%%%%%%
%\title{Symmetries of sub-Riemannian surfaces} %%%%%%%%%% ERASE FOR ARCHIVE %%%%%%%%%%%%
%\author[JR]{José Ricardo Arteaga B.} %%%%%%%%%% ERASE FOR ARCHIVE %%%%%%%%%%%%
%\address[JR]{Universidad de los Andes, Bogotá, Colombia} %%%%%%%%%% ERASE FOR ARCHIVE
%\author[MA]{Mikhail Armenovich Malakhaltsev}  %%%%%%%%%% ERASE FOR ARCHIVE %%%%%%%%%%%%
%\address[MA]{Kazan State University, Kazan, Russia} %%%%%%%%%% ERASE FOR ARCHIVE 
%\ead[JR]{jarteaga@uniandes.edu.co} %	warning! does not work correctly
%\ead[MA]{Mikhail.Malakhaltsev@ksu.ru} % warning! does not work correctly

\begin{abstract}
We obtain some results on symmetries of sub-Riemannian surfaces. In case of contact sub-Riemannian surface we base on invariants found by Hughen \cite{Hughen}. Using these invariants, we find conditions under which a sub-Riemannian surface does not admit symmetries. If a surface admits symmetries, we show how invariants help to find them. It is worth noting, that the obtained conditions can be explicitly checked for a given contact sub-Riemannian surface. Also, we consider sub-Riemannian surfaces which are not contact and find their invariants along the surface where the distribution fails to be contact.  
\end{abstract}
%\begin{keyword} %%%%%%%%%% ERASE FOR ARCHIVE %%%%%%%%%%%%
%Nonholonomic mechanics \sep equivalence method \sep G-structures \sep Sub-Riemannian geometry \sep Vakonomic dynamics  %%%%%%%%%% ERASE FOR ARCHIVE %%%%%%%%%%%%
%\PACS 70G45 \sep 53C10 \sep 37J60  %%%%%%%%%% ERASE FOR ARCHIVE %%%%%%%%%%%%
%\MSC[2008] 58A17 \sep 53C05 \sep 53C17 \sep 58E25  %%%%%%%%%% ERASE FOR ARCHIVE %%%%%%%%%%%%
%% keywords here, in the form: keyword \sep keyword
%% PACS codes here, in the form: \PACS code \sep code
%% MSC codes here, in the form: \MSC code \sep code
%% or \MSC[2008] code \sep code (2000 is the default
%\end{keyword}  %%%%%%%%%% ERASE FOR ARCHIVE %%%%%%%%%%%%
%\date{June, 2009} %%%%%%%%%% ERASE FOR ARCHIVE %%%%%%%%%%%%
%\end{frontmatter}  %%%%%%%%%% ERASE FOR ARCHIVE %%%%%%%%%%%%
%----------------------------

%\begin{linenumbers} %%%%%%%%%% ERASE WHEN YOU HAD FINISHED %%%%%%%%%%%%

\section*{Introduction}
A sub-Riemannian manifold is a $k$-dimensional distribution endowed by a metric tensor on an $n$-dimensional manifold. At present sub-Riemannian geometry is intensively studied, this is motivated by applications in various fields of science (see, e.\.g. the book \cite{Montgomery}, where many applications of sub-Riemannian geometry are presented; also, for interesting examples, we refer the reader to \cite{Bloch1}, \cite{Pavlov}, \cite{Schempp}, where applications to mechanics, thermodynamics, and biology are given). At the same time,  various aspects of the theory of symmetries of sub-Riemannian manifolds are widely investigated because symmetries are always of great importance for applications \cite{Bloch2}, \cite{Olver}. Many papers are devoted to the theory of homogeneous (in part, symmetric) sub-Riemannian manifolds (see e.\,g. \cite{Falbel1}, \cite{Falbel2}, \cite{Hughen}, \cite{Sachkov}). The main investigation tool in these papers is the Lie algebras theory as is usual when we study homogeneous spaces.

In the present paper we study symmetries of sub-Riemannian surfaces, i.\.e. of sub-Riemannian manifolds with $k=2$ and $n=3$. Our main goal is to give a practical tool (or an algorithmic procedure) for investigation of symmetries of a sub-Riemannian surface. The paper is organized as follows. In the first section we give in details construction of invariants of a contact sub-Riemannian surface using the Cartan reduction procedure (here we follow \cite{Hughen}) and show how to calculate them. In the second section we demonstrate how to apply invariants to finding symmetries of a contact sub-Riemannian surface. Finally, in the third section we consider a sub-Riemannian surface without assumption that it is contact and find invariants along the ``singular surface'', where the distribution fails to be contact.

\section{Contact sub-Riemannian surfaces}

Let $M$ be an $n$-dimensional manifold and $\Delta$ be a $k$-dimensional distribution on $M$ endowed by a metric tensor field 
\begin{equation}
\forall p \in M, \quad \langle \cdot,\cdot \rangle_p : \Delta_p \times \Delta_p \to \mathbb{R}.
\label{<eq:0_1>}
\end{equation}
Then $(M,\Delta,\langle \cdot,\cdot \rangle)$ is called a \textit{sub-Riemannian manifold} \cite{Montgomery}.

In the present paper we consider a \textit{sub-Riemannian surface} 
$\mathcal S = (M,\Delta,\langle \cdot,\cdot \rangle)$, i.\,e. a two-dimensional distribution $\Delta$ on 
a three-dimensional manifold $M$, where $\Delta$ is endowed by a metric tensor field $\langle \cdot,\cdot \rangle$. In addition, we assume that \textit{the distribution $\Delta$ and the manifold $M$ are oriented}. 
Note that we do not suppose that any metric on $M$ is given.

Throughout the paper we will denote the Lie algebra of vector fields on a manifold $N$ by $\mathfrak{X}(N)$, and the space of covector fields by $\mathfrak{X}(N)^*$. Also the space of $r$-forms on $N$ will be denoted by $\Lambda^r(N)$.

\subsection{$G$-structure associated with a sub-Riemannian surface}

\subsubsection{Elements of theory of  $G$-structures}
\label{subsubsec:g_structures}
Recall notions and results of the theory of $G$-structures we use in the present paper (for the details we refer the reader to \cite{Montgomery} and \cite{KN}).

\paragraph{Tautological forms, pseudoconnection form, and structure equations} 
Let $M$ be a smooth $n$-dimensional manifold, and $\pi : B(M) \to M$ be the coframe bundle of $M$.

On $B(M)$  the \textit{tautological forms} $\theta^a \in \Omega^1(B(M))$  are defined as follows  \cite{KN}. 
For a point $\xi \in B(M)$ ($\xi=\{\xi^a\}_{a=\overline{1,n}}$ is a coframe of $T_p M$, where $p = \pi(\xi)$), we set  
\begin{equation}
\theta^a_\xi : T_\xi (B(M)) \to \mathbb{R}, \quad \theta^a_\xi(X) = \xi^a(d\pi(X)).
\label{eq:1_11}
\end{equation}

Now, on a neighborhood $U$ of a point $p \in M$, take a coframe field $\eta = \{\eta^a\}$. This gives a trivialization  $\alpha : \pi^{-1}(U) \to U \times GL(n)$: 
to a coframe $\xi$ at $p \in U$ we assign  $(p,g) \in U \times GL(n)$ such that $\xi^a = \tilde g^a_b \eta^b_p$, where $||\tilde g^a_b || = g^{-1}$. 

For a coframe field $\eta$ on $U$ let us consider the pullback $1$-forms $\bar\eta^a = d\pi^*\eta^a$ on 
$U \times GL(n) \cong \pi^{-1}(U) \subset B(M)$. Then 
\begin{equation}
\theta^a_{(p,g)} = \tilde g^a_b \bar\eta^b_{(p,g)} = \tilde g^a_b d\pi^* \eta^b_p.
\label{eq:1_12}
\end{equation} 

A \textit{$G$-structure} $P \to M$ is a principal subbundle of $\pi : B(M) \to M$ with structure group $G \subset GL(n)$. The tautological forms on $P$ are the restrictions of $\theta^a$ to $P$ and will be denoted by the same letters.  

Let us denote by $\mathfrak{g}$ the Lie algebra of the Lie group $G$.
A \textit{pseudoconnection form} $\omega$ on a $G$-structure $\pi : P \to M$ is a $\mathfrak{g}$-valued 1-form on $P$ such that $\omega(\sigma(a))=a$, where $\sigma(a)$ is the fundamental vector field (\cite{KN}, Ch. I, Sec. 5) on $P$ corresponding to $a \in \mathfrak{a}$ .

Given a pseudoconnection form $\omega$, we have \textit{structure equations} on $P$:
\begin{equation}
d\theta^a = \omega^a_b \wedge \theta^b + T^a_{bc} \theta^b \wedge \theta^c
\label{eq:1_13}
\end{equation}
where the functions $T^a_{bc} : P \to \mathbb{R}$ uniquely determined by equations \eqref{eq:1_13} are called \textit{torsion functions}, and the map 
$T : P \to \Lambda^2 \mathbb{R}^n \otimes \mathbb{R}^n$, $ \xi \to \{T^a_{bc}(\xi)\}$, is called the \textit{torsion} of the pseudoconnection $\omega^a_b$.

\paragraph{Structure function}
Let us find how the torsion changes under change of the pseudoconnection.
If $\omega^a_b$, $\hat\omega^a_b$ are pseudoconnections on $P$, then $\mu^a_b = \hat{\omega}^a_b-\omega^a_b$ is a $\mathfrak{g}$-valued form on $P$ with property that $\mu(\sigma(a)) = 0$ for any $a \in \mathfrak{g}$. Then  $\mu^a_b = \mu^a_{bc} \theta^c$.  

\begin{multline}
d\theta^{a} = \hat\omega^{a}_{b} \wedge \theta^{b} + \hat T^{a}_{bc} \theta^{b} \wedge \theta^{c} = 
\left (\omega^{a}_{b} + \mu^{a}_{bc}\theta^{c}\right )\wedge \theta^{b} + \hat T^{a}_{bc} \theta^{b} \wedge \theta^{c} = 
\\
\omega^{a}_{b} \wedge \theta^{b} + \left (\hat T^{a}_{bc} - \mu^{a}_{[bc]}\right ) \theta^{b} \wedge \theta^{c} =
\omega^{a}_{b}\wedge \theta^{b} + T^{a}_{bc} \theta^{b} \wedge \theta^{c}
\end{multline}
	Hence follows that 
\begin{equation}
\hat{\omega}^a_b=\omega^a_b  + \mu^a_{bc}\theta^c \Rightarrow \hat T^{a}_{bc} = T^{a}_{bc}  + \mu^{a}_{[bc]}
\label{eq:1_31}
\end{equation}
Let us define the Spencer operator $\delta$ from the space of tensors $T^2_1(\mathbb{R}^n)$ of type $(2,1)$ to the space $\Lambda^2(\mathbb{R}^n) \otimes \mathbb{R}^n$ as follows:   
\begin{equation}
\delta : t^a_{bc} \in T^2_1(\mathbb{R}^n) \mapsto t^a_{[bc]} = \frac{1}{2}(t^a_{bc} - t^a_{cb}).
\label{eq:1_32}
\end{equation}
Note that $\mathfrak{g}\otimes(\mathbb{R}^n)^* \subset \mathfrak{gl}(n) \otimes \mathbb{R}^* \cong T^2_1(\mathbb{R}^n)$ and we will denote the restriction of $\delta$ to $\mathfrak{g}\otimes(\mathbb{R}^n)^*$ by the same letter $\delta$.
Thus, \eqref{eq:1_31} can be rewritten as follows:
\begin{equation}
\hat{\omega}^a_b=\omega^a_b  + \mu^a_{bc}\theta^c \Rightarrow \hat T^{a}_{bc} = T^{a}_{bc}  + \delta(\mu^{a}_{bc}).
\label{eq:1_32_1}
\end{equation}
From \eqref{eq:1_32_1} we conclude that \textit{if 
$\delta: \mathfrak{g}\otimes(\mathbb{R}^n)^* \to \Lambda^2(\mathbb{R}^n) \otimes \mathbb{R}^n$ is a monomorphism, then, pseudoconnections $\omega^a_b$, $\hat\omega^a_b$ with the same torsion $T^a_{bc}$ coincide}. 

Now denote
\begin{equation}
\mathcal{T} = \frac{\Lambda^2(\mathbb{R}^n) \otimes \mathbb{R}^n}{\delta(\mathfrak{g}\otimes(\mathbb{R}^n)^*)}.
\label{eq:1_33}
\end{equation}
From \eqref{eq:1_32_1}  it follows that one can correctly define the \textit{structure function}:
\begin{equation}
\mathcal{C} : P \to \mathcal{T},
\quad 
\xi \mapsto [T^a_{bc}(\xi)].
\label{eq:1_34}
\end{equation}
\paragraph{$G$-equivariance of structure function}
The group $G$ acts on $\Lambda^2(\mathbb{R}^n) \otimes \mathbb{R}^n$ from the right as follows:
\begin{equation}
(\bar\rho(g)T)^a_{bc} = \tilde g^a_r T^r_{pq} g^p_b g^q_c
\label{eq:1_35}
\end{equation}
and one can easily prove that the subspace $\delta(\mathfrak{g}\otimes(\mathbb{R}^n)^*)$ is invariant under this action. Then  we have the following $G$-action on $\mathcal{T}$:  
\begin{equation}
\forall g \in G, \quad  \rho(g) : 
\mathcal{T} \to \mathcal{T},
\quad 
[T^a_{bc}] \mapsto  [\tilde g^a_r T^r_{pq} g^p_b g^q_c]
\label{eq:1_36}
\end{equation}
By cumbersome calculations, from the structure equations \eqref{eq:1_13} one can obtain that 
\begin{equation}
\mathcal{C}(\xi g) = \mathcal{C}(g^{-1}\xi) = \rho(g) \mathcal{C}(\xi), \forall \xi \in P, g \in G.
\label{eq:1_37}
\end{equation} 
\begin{remark}
If $\omega$ is a connection, one can prove that $T^a_{bc} (\xi g) = \tilde g^a_r T^r_{pq}(\xi) g^p_b g^q_c$, however it is wrong if $\omega$ is a pseudoconnection. In this case, we have only that   
$T^a_{bc} (\xi g) = \tilde g^a_r T^r_{pq}(\xi) g^p_b g^q_c + \nu^a_{bc}$, where $\nu^a_{bc} \in \delta(\mathfrak{g}\otimes(\mathbb{R}^n)^*)$.
\end{remark}

\subsubsection{Cartan reduction}
\label{subsubsec:cartan_reduction}
Let $P \to M$ be a $G$-structure.  Let $\mathcal{T} = \sqcup \mathcal{T}_\alpha$ be the decomposition  of $\mathcal{T}$ into orbits of  the $G$-action \eqref{eq:1_36}. Assume that the structure function $c$ takes values in one orbit $\mathcal{T}_0$, only. 

Fix $\tau_0 \in \mathcal{T}_0$. Then 
\begin{equation}
P_1 = \{\xi \mid  \mathcal{C}(\xi) = \tau_0\}
\label{eq:1_38}
\end{equation}
is the total space of a principal $G_1$-subbundle of $P$, where 
\begin{equation}
G_1 = \{ g \in G \mid \rho(g)\tau_0 = \tau_0 \}.
\label{eq:1_39}
\end{equation}

They say that the $G_1$-structure $P_1 \to M$ is obtained by the \textit{Cartan reduction} from the $G$-structure $P \to M$.

\subsection{$G$-structure associated to a sub-Riemannian surface}

Let $\mathcal S = (M,\Delta,\langle \cdot,\cdot \rangle)$ be a sub-Riemannian surface. 
%For a $3$-dimensional manifold $M$ let us denote by $B(M)$ the principal $GL(3)$-bundle of coframes of $M$.
We say that a coframe $\eta = \left(\eta^{1}, \eta^{2}, \eta^{3} \right)$ of $T_p M$, $p \in M$, is adapted to $\mathcal S$ if
\begin{itemize}
\item[(1)]%
$\eta$ is positively oriented, and $(\eta^1 |_{\Delta_p}, \eta^2 |_{\Delta_p})$ is a positively oriented coframe of $\Delta_p$;
\item[(2)]%
$\eta^3 \in Ann(\Delta)_p$, or, equivalently, $\eta^3 (W) = 0$ for any $W \in \Delta_p$; 
\item[(3)]%
$ \langle W , W \rangle = [\eta^{1}(W)]^{2} + [\eta^{2}(W)]^2$ for any $W \in \Delta_p$.
\end{itemize}

  To a given sub-Riemannian surface  $\mathcal S = (M,\Delta,\langle \cdot,\cdot \rangle)$ we associate the principal subbundle $B_0 \subset B$ consisting of adapted frames. It is clear that the structure group of $B_0$ is  
\begin{equation}\nonumber
G_{0}=
\left \{
\left (
\begin{array}{cc}
A & B \\
0 & c
\end{array}
\right )\mid A \in SO(2), B=
\left (
\begin{array}{cc}
b_{1} \\
b_{2}
\end{array}
\right )
\in \mathbf{R}^{2},
c\in \mathbf{R}\backslash \left \{0\right \}
\right \}.
\end{equation}
One can easily prove

\begin{proposition}
\label{prop:1_1}
A sub-Riemannian surface $\mathcal S = (M,\Delta,\langle \cdot,\cdot \rangle)$ is equivalent to a $G_0$-structure on $M$.
\end{proposition}

\subsubsection{Contact sub-Riemannian surfaces}
\paragraph{Contact distributions}

Let $\omega$ be a $1$-form on a $(2n+1)$-dimensional manifold $M$. The form $\omega$ is said to be \textit{contact} if 
\begin{equation}
 \underbrace{\omega \wedge d\omega \wedge \ldots \wedge d\omega}_n \ne 0.
\label{eq:1_300}
\end{equation}
A $1$-form is contact if and only  if $\omega$ is nonvanishing (hence the Pfaff equation $\omega=0$ determines a $2n$-dimensional distribution $\Delta$), and $d\omega |_\Delta$ is nondegenerate. 

Let $\Delta$ be a $2n$-dimensional distribution  on a $(2n+1)$-dimensional manifold $M$. 
Denote by $Ann(\Delta)$ the vector subbundle in $T^* M$ of rank $1$ such that the fiber of $Ann(\Delta)$ at $p \in M$ is 
\begin{equation}
Ann(\Delta)_p = \{\omega \in T^*_p M \mid \omega(W) = 0 \ \forall W \in \Delta_p\}.
\label{eq:1_1}
\end{equation} 
The distribution $\Delta$ is said to be \textit{contact} if for each $p \in M$ there exists a contact section $\omega$ of $Ann(\Delta)$ in a neighborhood of $p$. Note that this definition does not depend on the choice of $\omega$: \textit{if $\Delta$ is contact, then any nonvanishing section of $Ann(\Delta)$ is contact}.

We say that a sub-Riemannian surface $\mathcal S = (M,\Delta,\langle \cdot,\cdot \rangle)$ is \textit{contact} if the distribution $\Delta$ is contact.

%----------------------% Theorem 1
\begin{theorem}\label{Th:Theorema-1}
Any contact sub-Riemannian surface $\mathcal S = (M,\Delta,\langle \cdot,\cdot \rangle)$ uniquely determines an $SO(2)$-structure $B_{2}\to M$ and a connection  on $B_2$ with the connection form  
\begin{equation}
\omega = 
\left (
\begin{array}{ccc}
0 & \alpha & 0 \\
-\alpha & 0 & 0 \\
0 & 0 & 0 
\end{array}
\right ).
\label{eq:1_14}
\end{equation}
The $1$-form $\alpha$ in \eqref{eq:1_14} together with the tautological forms $\theta^a$ give a coframe field on $B_2$. The structure equations \eqref{eq:1_13} are written as follows: 
\begin{equation}
\left (
\begin{array}{c}
d\theta^{1}\\
d\theta^{2}\\
d\theta^{3}
\end{array}
\right )
=
\left (
\begin{array}{ccc}
0 & \alpha & 0 \\
-\alpha & 0 & 0 \\
0 & 0 & 0 
\end{array}
\right )
\wedge
\left (
\begin{array}{c}
\theta^{1}\\
\theta^{2}\\
\theta^{3}
\end{array}
\right )
+
\left (
\begin{array}{ccc}
a_{1} & a_{2} & 0 \\
a_{2} & -a_{1} & 0 \\
0 & 0 & 1 
\end{array}
\right )
\left (
\begin{array}{c}
\theta^{2}\wedge\theta^{3}\\
\theta^{3}\wedge\theta^{1}\\
\theta^{1}\wedge\theta^{2}
\end{array}
\right )
\label{eq:1_15}
\end{equation}
\end{theorem}

This theorem was proved by K. Hughen in \cite{Hughen} (also a sketch of the proof is given in  \cite{Montgomery}, Ch.~7, 7.10). We present here a detailed proof of this theorem which is based on the Cartan reduction procedure as it was exposed in \ref{subsubsec:g_structures} and  \ref{subsubsec:cartan_reduction}.

\subsubsection{Proof of Theorem~\ref{Th:Theorema-1}}

\textbf{Step 1}. We start with the $G_0$-structure $B_0 \to M$ associated with $\mathcal{S}$ (see Proposition~\ref{prop:1_1}). We have 
\begin{equation}
G_{0}=\left \{
\left.
\left(
\begin{array}{ccc}
\cos\varphi & -\sin\varphi & b_{1}\\
\sin\varphi & \cos\varphi & b_{2} \\
0 & 0 & c  
\end{array}
\right)
\ \right| c \ne 0
\right\}
\label{eq:1_201}
\end{equation}
Then the Lie algebra of $G_0$ is
\begin{equation}
\mathfrak{g}_{0}=\left \{
\left.
\left (
\begin{array}{ccc}
0 & \alpha & \beta_{1}\\
-\alpha & 0 & \beta_{2} \\
0 & 0 & \gamma  
\end{array}
\right )
\ \right|
\alpha, \beta_1, \beta_2, \gamma \in \mathbb{R}
\right\}
\label{eq:1_202}
\end{equation}
Now we will calculate 
\begin{equation}
\mathcal{T}_0 = 
\frac{\Lambda^2(\mathbb{R}^3) \otimes \mathbb{R}^3}{\delta(\mathfrak{g}_0\otimes(\mathbb{R}^3)^*)},
\label{eq:1_203}
\end{equation}
 (see \eqref{eq:1_33}). Let us denote the standard basis of $\mathbb{R}^3$ by $e_1$, $e_2$, $e_3$, and the dual basis by $e^1$, $e^2$, $e^3$. Then the basis of $\mathfrak{g}_0$ is 
\begin{equation}
\mathcal{E}_1 = e_1 \otimes e^2 - e_2 \otimes e^1;
\mathcal{E}_2 = e_1 \otimes e^3;
\mathcal{E}_3 = e_2 \otimes e^3;
\mathcal{E}_4 = e_3 \otimes e^3;
\label{eq:1_204}
\end{equation}
and $\mathcal{E}_i \otimes e^a$, $i=\overline{1,4}$, $a=\overline{1,3}$, is the basis of $\mathfrak{g}_0 \otimes (\mathbb{R}^3)^*$. 
Then
$\delta: \mathfrak{g}_0 \otimes(\mathbb{R}^n)^* \to \Lambda^2(\mathbb{R}^n) \otimes \mathbb{R}^n$ acts on the basis elements as follows
\begin{equation}
\begin{cases}
\mathcal{E}_1 \otimes e^1 \mapsto - e_1 \otimes e^2 \wedge  e^1,
\quad
\mathcal{E}_1 \otimes e^2 \mapsto  - e_2 \otimes e^1 \wedge  e^2,
\\
\mathcal{E}_1 \otimes e^3 \mapsto   e_1 \otimes e^2 \wedge  e^3 - e_2 \otimes e^1 \wedge e^3,
\\
\mathcal{E}_2 \otimes e^1 \mapsto   e_1 \otimes e^3 \wedge  e^1,
\quad
\mathcal{E}_2 \otimes e^2 \mapsto   e_1 \otimes e^3 \wedge  e^2,
\quad
\mathcal{E}_2 \otimes e^3 \mapsto  0, 
\\
\mathcal{E}_3 \otimes e^1 \mapsto   e_2 \otimes e^3 \wedge  e^1,
\quad
\mathcal{E}_3 \otimes e^2 \mapsto   e_2 \otimes e^3 \wedge  e^2,
\quad
\mathcal{E}_3 \otimes e^3 \mapsto  0, 
\\
\mathcal{E}_4 \otimes e^1 \mapsto   e_3 \otimes e^3 \wedge  e^1,
\quad
\mathcal{E}_4 \otimes e^2 \mapsto   e_3 \otimes e^3 \wedge  e^2,
\quad
\mathcal{E}_4 \otimes e^3 \mapsto  0, 
\end{cases} 
\label{eq:1_205}
\end{equation}
From \eqref{eq:1_205} we get that $\mathcal{T}_0$ is spanned by $[e_3 \otimes e^1 \wedge e^2]$ and so is one-dimensional. 

Now let us find the action of $G_0$ on $\mathcal{T}_0$. From \eqref{eq:1_36} we get that, for any $g \in G_0$, 
\begin{equation}
\rho(g) [e_3 \otimes e^1 \wedge e^2] = [ \tilde g^a_3 g^1_b g^2_c e_a \otimes e^b \wedge e^c] = [\tilde g^3_3 (g^1_1 g^2_2 - g^2_1 g^1_2) e_3 \otimes e^1 \wedge e^2 ] = c^{-1} [e_3 \otimes e^1 \wedge e^2].
\label{eq:1_206}
\end{equation}
Hence the action of $G_0$ on $\mathcal{T}_0$ has two orbits: $\mathcal{O}_0 = \{0 \in \mathcal{T}_0\}$ and $\mathcal{O}_1 = \{t \in \mathcal{T}_0 \mid t \ne 0\}$. 

Let us prove that, if $\mathcal{S}$ is contact, the structure function $\mathcal{C}$ takes values in $\mathcal{O}_1$.
The structure equations \eqref{eq:1_13}  can be written as follows :
\begin{equation}
\left (
\begin{array}{c}
d\theta^{1}\\
d\theta^{2}\\
d\theta^{3}
\end{array}
\right )
=
\left (
\begin{array}{ccc}
0 & \alpha &\beta \\
-\alpha & 0 & \gamma \\
0 & 0 & \delta 
\end{array}
\right )
\wedge
\left (
\begin{array}{c}
\theta^{1}\\
\theta^{2}\\
\theta^{3}
\end{array}
\right )
+
\left (
\begin{array}{ccc}
T_{23}^{1} & T_{31}^{1} &T_{12}^{1} \\
T_{23}^{2} & T_{31}^{2} &T_{12}^{2} \\
T_{23}^{3} & T_{31}^{3} &T_{12}^{3} 
\end{array}
\right )
\left (
\begin{array}{c}
\theta^{2}\wedge\theta^{3}\\
\theta^{3}\wedge\theta^{1}\\
\theta^{1}\wedge\theta^{2}
\end{array}
\right )
\label{eq:1_207}
\end{equation}
Now take a section $s : U \to B_0$, that is a coframe field $\eta^a$ adapted to $\mathcal{S}$ on $U$. Then 
from the definition of the tautological forms we have that $ds^* \theta^a = \eta^a$, hence  
$ds^*(d\theta^3 \wedge \theta^3) = d\eta^3 \wedge \eta^3 \ne 0$ because $\eta^3$ is a contact form (see \eqref{eq:1_300}). At the same time, from \eqref{eq:1_207}, we get that 
$d\theta^3 \wedge \theta^3 = T^3_{12} \theta^1 \wedge \theta^2 \wedge \theta^3$, hence follows that $T^3_{12} \ne 0$.
Thus $\mathcal{C}(s(p)) = [T^3_{12}(s(p)) e_3 \otimes e^1 \wedge e^2] \ne 0$ for any $p \in U$. As each $\xi \in \pi^{-1}(U)$ can be written as $\xi = s(p) g$, $p = \pi(\xi)$, and the structure function $\mathcal{C}$ satisfies \eqref{eq:1_37}, we have that $\mathcal{C}(\xi) \ne 0$ for any $\xi \in B_0$, hence $\mathcal{C}$ takes values in $\mathcal{O}_1$. 

Thus, we can make the Cartan reduction and pass to the $G_1$-structure $B_1 \to M$, where 
\begin{equation}
B_1 = \{\xi \mid  \mathcal{C}(\xi) = [e_3 \otimes e^1 \wedge e^2]\}
\label{eq:1_208}
\end{equation}
is the total space of a principal $G_1$-subbundle of $B_0$, and 
\begin{equation}
G_1 = \{ g \in G \mid \rho(g)[e_3 \otimes e^1 \wedge e^2] = [e_3 \otimes e^1 \wedge e^2] \} 
=\left \{
\left(
\begin{array}{ccc}
\cos\varphi & -\sin\varphi & b_{1}\\
\sin\varphi & \cos\varphi & b_{2} \\
0 & 0 & 1  
\end{array}
\right)
\right\}
\label{eq:1_209}
\end{equation}

\textbf{Step 2}.
The Lie algebra of $G_1$ is
\begin{equation}
\mathfrak{g}_{1}=\left \{
\left.
\left (
\begin{array}{ccc}
0 & \alpha & \beta_{1}\\
-\alpha & 0 & \beta_{2} \\
0 & 0 & 0 
\end{array}
\right )
\ \right|
\alpha, \beta_1, \beta_2,  \in \mathbb{R}
\right\}
\label{eq:1_210}
\end{equation}
and, by the construction of $B_1$, the structure equations have the form:
\begin{equation}
\left (
\begin{array}{c}
d\theta^{1}\\
d\theta^{2}\\
d\theta^{3}
\end{array}
\right )
=
\left (
\begin{array}{ccc}
0 & \alpha &\beta \\
-\alpha & 0 & \gamma \\
0 & 0 & 0 
\end{array}
\right )
\wedge
\left (
\begin{array}{c}
\theta^{1}\\
\theta^{2}\\
\theta^{3}
\end{array}
\right )
+
\left (
\begin{array}{ccc}
T_{23}^{1} & T_{31}^{1} &T_{12}^{1} \\
T_{23}^{2} & T_{31}^{2} &T_{12}^{2} \\
T_{23}^{3} & T_{31}^{3} & 1 
\end{array}
\right )
\left (
\begin{array}{c}
\theta^{2}\wedge\theta^{3}\\
\theta^{3}\wedge\theta^{1}\\
\theta^{1}\wedge\theta^{2}
\end{array}
\right )
\label{eq:1_210_1}
\end{equation}
With notation of the previous step we have  the basis of $\mathfrak{g}_1$ is 
\begin{equation}
\mathcal{E}_1 = e_1 \otimes e^2 - e_2 \otimes e^1;
\mathcal{E}_2 = e_1 \otimes e^3;
\mathcal{E}_3 = e_2 \otimes e^3;
\label{eq:1_211}
\end{equation}
and $\mathcal{E}_i \otimes e^a$, $i=\overline{1,3}$, $a=\overline{1,3}$, is the basis of $\mathfrak{g}_1 \otimes (\mathbb{R}^3)^*$. 
Then
$\delta: \mathfrak{g}_1 \otimes(\mathbb{R}^n)^* \to \Lambda^2(\mathbb{R}^n) \otimes \mathbb{R}^n$ acts on the basis elements as follows
\begin{equation}
\begin{cases}
\mathcal{E}_1 \otimes e^1 \mapsto - e_1 \otimes e^2 \wedge  e^1,
\quad
\mathcal{E}_1 \otimes e^2 \mapsto  - e_2 \otimes e^1 \wedge  e^2,
\\
\mathcal{E}_1 \otimes e^3 \mapsto   e_1 \otimes e^2 \wedge  e^3 - e_2 \otimes e^1 \wedge e^3,
\\
\mathcal{E}_2 \otimes e^1 \mapsto   e_1 \otimes e^3 \wedge  e^1,
\quad
\mathcal{E}_2 \otimes e^2 \mapsto   e_1 \otimes e^3 \wedge  e^2,
\quad
\mathcal{E}_2 \otimes e^3 \mapsto  0, 
\\
\mathcal{E}_3 \otimes e^1 \mapsto   e_2 \otimes e^3 \wedge  e^1,
\quad
\mathcal{E}_3 \otimes e^2 \mapsto   e_2 \otimes e^3 \wedge  e^2,
\quad
\mathcal{E}_3 \otimes e^3 \mapsto  0, 
\end{cases} 
\label{eq:1_212}
\end{equation}
From \eqref{eq:1_212} we get that
\begin{equation}
\mathcal{T}_1 = 
\frac{\Lambda^2(\mathbb{R}^3) \otimes \mathbb{R}^3}{\delta(\mathfrak{g}_1\otimes(\mathbb{R}^3)^*)},
\label{eq:1_213}
\end{equation}
is spanned by $[e_3 \otimes e^2 \wedge e^3]$, $[e_3 \otimes e^3 \wedge e^1]$, $[e_3 \otimes e^1 \wedge e^2]$ and so is three-dimensional. 

At the same time, by construction of $B_1$, the structure function $\mathcal{C}$ takes values in the affine subspace 
\begin{equation}
\mathcal{T}'_1 = \{ u [e_3 \otimes e^2 \wedge e^3] + v [e_3 \otimes e^3 \wedge e^1] + [e_3 \otimes e^1 \wedge e^2] \} \subset \mathcal{T}_1
\label{eq:1_214}
\end{equation}

Let us find the action of $G_1$ on $\mathcal{T}_1$. Using \eqref{eq:1_36}, we find that
\begin{eqnarray*}
\rho(g) [e_3 \otimes e^2 \wedge e^3] &=& \cos\varphi [e_3 \otimes e^2 \wedge e^3] - \sin\varphi [e_3 \otimes e^3 \wedge e^1]
\\
\rho(g) [e_3 \otimes e^3 \wedge e^1] &=& \sin\varphi [e_3 \otimes e^2 \wedge e^3] + \cos\varphi [e_3 \otimes e^3 \wedge e^1]
\\
\rho(g) [e_3 \otimes e^1 \wedge e^2] &=& (-b_1 \cos\varphi - b_2 \sin\varphi) [e_3 \otimes e^2 \wedge e^3] +
\\
&& (b_1 \sin\varphi - b_2 \cos\varphi)[e_3 \otimes e^3 \wedge e^1] + [e_3 \otimes e^1 \wedge e^2]
\end{eqnarray*} 
From this follows that $\rho(g)$ maps $\mathcal{T}'_1$ into itself, and moreover, 
\begin{multline}
\rho(g) \big(u [e_3 \otimes e^2 \wedge e^3] + v [e_3 \otimes e^3 \wedge e^1] + [e_3 \otimes e^1 \wedge e^2]\big) = 
\\
\{(u-b_1)\cos\varphi+(v-b_2)\sin\varphi\}[e_3 \otimes e^2 \wedge e^3] + 
\\
\{-(u-b_1)\sin\varphi+(v-b_2)\cos\varphi\}[e_3 \otimes e^3 \wedge e^1] + [e_3 \otimes e^1 \wedge e^2]
\label{eq:1_215}
\end{multline}

Thus, we can make the Cartan reduction and pass to the $G_2$-structure $B_2 \to M$. We take 
\begin{equation}
\tau_1 =  0 \cdot [e_3 \otimes e^2 \wedge e^3] + 0 \cdot [e_3 \otimes e^3 \wedge e^1] + [e_3 \otimes e^1 \wedge e^2] \in \mathcal{T}_1
\label{eq:1_216}
\end{equation}
and set  
\begin{equation}
B_2 = \{\xi \mid  \mathcal{C}(\xi) = \tau_1\}
\label{eq:1_217}
\end{equation}
is the total space of a principal $G_2$-subbundle of $B_1$, and 
\begin{equation}
G_2 = \{ g \in G_1 \mid \rho(g)\tau_1 = \tau_1 \} 
=\left \{
\left(
\begin{array}{ccc}
\cos\varphi & -\sin\varphi & 0\\
\sin\varphi & \cos\varphi & 0 \\
0 & 0 & 1  
\end{array}
\right)
\right\}
\label{eq:1_218}
\end{equation}

\textbf{Step 3}.
The Lie algebra of $G_2$ is
\begin{equation}
\mathfrak{g}_{2}=\left \{
\left.
\left (
\begin{array}{ccc}
0 & \alpha & 0\\
-\alpha & 0 & 0 \\
0 & 0 & 0 
\end{array}
\right )
\ \right|
\alpha  \in \mathbb{R}
\right\}
\label{eq:1_219}
\end{equation}
and the structure function $\mathcal{C}(\xi) = \tau_1$, for any $\xi \in B_2$,  hence the structure equations are written as follows: 
\begin{equation}
\left (
\begin{array}{c}
d\theta^{1}\\
d\theta^{2}\\
d\theta^{3}
\end{array}
\right )
=
\left (
\begin{array}{ccc}
0 & \alpha &  0\\
-\alpha & 0 & 0 \\
0 & 0 & 0 
\end{array}
\right )
\wedge
\left (
\begin{array}{c}
\theta^{1}\\
\theta^{2}\\
\theta^{3}
\end{array}
\right )
+
\left (
\begin{array}{ccc}
T_{23}^{1} & T_{31}^{1} &T_{12}^{1} \\
T_{23}^{2} & T_{31}^{2} &T_{12}^{2} \\
0  & 0 & 1 
\end{array}
\right )
\left (
\begin{array}{c}
\theta^{2}\wedge\theta^{3}\\
\theta^{3}\wedge\theta^{1}\\
\theta^{1}\wedge\theta^{2}
\end{array}
\right )
\label{eq:1_220}
\end{equation}
The basis of $\mathfrak{g}_2$ is 
\begin{equation}
\mathcal{E}_1 = e_1 \otimes e^2 - e_2 \otimes e^1;
\label{eq:1_221}
\end{equation}
and $\mathcal{E}_1 \otimes e^a$, $a=\overline{1,3}$, is the basis of $\mathfrak{g}_2 \otimes (\mathbb{R}^3)^*$. 
Then
$\delta: \mathfrak{g}_2 \otimes(\mathbb{R}^3)^* \to \Lambda^2(\mathbb{R}^3) \otimes \mathbb{R}^3$ acts on the basis elements as follows
\begin{equation}
\begin{cases}
\mathcal{E}_1 \otimes e^1 \mapsto  e_1 \otimes e^1 \wedge  e^2,
\quad
\mathcal{E}_1 \otimes e^2 \mapsto  - e_2 \otimes e^1 \wedge  e^2,
\\
\mathcal{E}_1 \otimes e^3 \mapsto   e_1 \otimes e^2 \wedge  e^3 + e_2 \otimes e^3 \wedge e^1,
\end{cases} 
\label{eq:1_222}
\end{equation}
Hence follows immediately that $\delta$ \textit{is a monomorphism}.

From \eqref{eq:1_222} we get that
\begin{equation}
\mathcal{T}_2 = 
\frac{\Lambda^2(\mathbb{R}^3) \otimes \mathbb{R}^3}{\delta(\mathfrak{g}_2\otimes(\mathbb{R}^3)^*)},
\label{eq:1_223}
\end{equation}
is spanned by 
\begin{align*}
&&[e_1 \otimes e^2 \wedge e^3] - [e_2 \otimes e^3 \wedge e^1], [e_1 \otimes e^3 \wedge e^1], [e_2 \otimes e^2 \wedge e^3], 
\\
&&[e_3 \otimes e^2 \wedge e^3], [e_3 \otimes e^3 \wedge e^1], [e_3 \otimes e^1 \wedge e^2]
\label{eq:1_224}
\end{align*}
and so is six-dimensional. 

However,  by construction of $B_2$, the structure function $\mathcal{C}$ takes values in the affine subspace 
\begin{equation}
\begin{split}
\mathcal{T}'_2 = \{ u ([e_1 \otimes e^2 \wedge e^3] - [e_2 \otimes e^3 \wedge e^1]) + v [e_1 \otimes e^3 \wedge e^1] + 
\\
w [e_2 \otimes e^2 \wedge e^3] + [e_3 \otimes e^1 \wedge e^2] \} \subset \mathcal{T}_2
\label{eq:1_225}
\end{split}
\end{equation}
Hence follows that the structure equations have the following form 
\begin{equation}
\left (
\begin{array}{c}
d\theta^{1}\\
d\theta^{2}\\
d\theta^{3}
\end{array}
\right )
=
\left (
\begin{array}{ccc}
0 & \alpha &  0\\
-\alpha & 0 & 0 \\
0 & 0 & 0 
\end{array}
\right )
\wedge
\left (
\begin{array}{c}
\theta^{1}\\
\theta^{2}\\
\theta^{3}
\end{array}
\right )
+
\left (
\begin{array}{ccc}
u  &  v & 0 \\
w  & - u & 0 \\
0  & 0 & 1 
\end{array}
\right )
\left (
\begin{array}{c}
\theta^{2}\wedge\theta^{3}\\
\theta^{3}\wedge\theta^{1}\\
\theta^{1}\wedge\theta^{2}
\end{array}
\right )
\label{eq:1_226}
\end{equation}

Let us now prove that the form $\alpha$ which satisfies \eqref{eq:1_226} is unique, In $\Lambda^2(\mathbb{R}^3) \otimes \mathbb{R}^3$ consider subspace $N$ spanned by  
\begin{align*}
&&e_1 \otimes e^2 \wedge e^3 - e_2 \otimes e^3 \wedge e^1, e_1 \otimes e^3 \wedge e^1, e_2 \otimes e^2 \wedge e^3, 
\\
&&e_3 \otimes e^2 \wedge e^3, e_3 \otimes e^3 \wedge e^1, e_3 \otimes e^1 \wedge e^2.
\end{align*}
It is clear that we have the direct sum
\begin{equation}
\Lambda^2(\mathbb{R}^3) \otimes \mathbb{R}^3 = N \oplus \delta(\mathfrak{g}_2\otimes(\mathbb{R}^3)^*)
\label{eq:1_227}
\end{equation}
and $\{T^a_{bc}\}$ from \eqref{eq:1_226}  takes values in $N$,
If we have another $\hat\alpha$ and the corresponding torsion $\{\hat T^a_{bc}\}$ which satisfy \eqref{eq:1_226}, then $\{\hat T^a_{bc}\}$ also take values in $N$, so the same is true for  $\{\hat T^a_{bc} - T^a_{bc}\}$.
However, from \eqref{eq:1_31} it follows that $\hat T^a_{bc} - T^a_{bc}  = \delta(\mu^a_{bc})$. As we have the direct sum decomposition \eqref{eq:1_227}, we obtain that  $\hat T^a_{bc} - T^a_{bc} = 0$ and $\delta(\mu^a_{bc})=0$. But $\delta$ is a monomorphism (see \eqref{eq:1_222}), hence $\mu^a_{bc} = 0$, and so $\hat\omega^a_b = \hat\omega^a_b$ (see \eqref{eq:1_31}). Thus $\hat\alpha = \alpha$. 

To finish the proof of the theorem it is sufficient to prove that in \eqref{eq:1_226} $v = w$. 
From \eqref{eq:1_226} we get $d\theta^3 = \theta^1 \wedge \theta^2$, then, again using \eqref{eq:1_226}, we
obtain
\begin{equation}
0 = d\theta^1 \wedge \theta^2 - \theta^1 \wedge d\theta^2 = v \theta^3 \wedge \theta^1 \wedge \theta^2 - w \theta^1 \wedge \theta^2 \wedge \theta^3 = (v-w)\theta^1 \wedge \theta^2 \wedge \theta^3. 
\label{eq:1_228}
\end{equation}
Now we set $u = a_1$, $v = w = a_2$ and from \eqref{eq:1_226} get the structure equations \eqref{eq:1_15}.

Thus we have proved that for any contact sub-Riemannian surface $\mathcal S = (M,\Delta,\langle \cdot,\cdot \rangle)$ there exists an $SO(2)$-structure on $M$ and a unique pseudoconnection form $\alpha$ such that    
the structure equations \eqref{eq:1_15} hold true. The uniqueness of the $SO(2)$-structure $B_2$ will be proved later, in Corollary~\ref{cor:1} of Proposition~\ref{prop:2}.

\subsubsection{The functions $a_1$, $a_2$ and $1$-form $\alpha$ in terms of structure functions of a local frame}

Let $\eta = \{\eta^a\}$ be a coframe field in a neighborhood $U$ of $p \in M$ which is a section of $B_2 \to M$. Let $d\eta^a = C^a_{bc} \eta^b \wedge \eta^c$ be the corresponding structure equations. Then, for $\bar\eta^a = d\pi^*\eta^a$, we have $d\bar{\eta}^a = \bar{C}^a_{bc}  \bar{\eta}^b \wedge \bar{\eta}^c$, where $\bar{C}^a_{bc} = \pi^* C^a_{bc} = C^a_{bc} \circ \pi : \pi^{-1}(U) \to \mathbb{R}$.

Let $\psi : \pi^{-1}(U) \to U \times SO(2)$ be a local trivialization of $\pi : B_2 \to M$ determined by $\eta$, then  
\begin{equation}
\psi^{-1} (p,g(\varphi)) = g(\varphi)^{-1} \eta_p,  
\end{equation}
where 
\begin{equation}
g(\varphi)=
\left (
\begin{array}{ccc}
\cos\varphi & -\sin\varphi & 0 \\
\sin\varphi & \cos\varphi & 0 \\
0 & 0 & 1
\end{array}
\right ),\qquad
\eta_p  = \left( \begin{array}{c} \eta^1_p \\ \eta^2_p \\ \eta^3_p \end{array} \right).
\label{eq:1_16}
\end{equation}

\begin{proposition}
\label{prop:2}
\emph{a)} If a coframe field $\eta = \{\eta^a\}$ is a local section of $B_2 \to M$, then $d\eta^3 = \eta^1 \wedge \eta^2$.  

\emph{b)} The form $\alpha$ is expressed in terms of $\bar C^a_{bc}$ as follows: 
\begin{equation}
\label{eq:1_90}
\alpha = d\varphi + \bar C^{1}_{12}\bar\eta^{1}+\bar C^{2}_{12}\bar\eta^{2}-\frac{1}{2}\left (\bar C^{1}_{23}+\bar C^{2}_{31}\right )\bar\eta^{3}
\end{equation}

\emph{c)} The functions $a_{1}$ and $a_{2}$ are expressed in terms of $\bar C^a_{bc}$ as follows:
\begin{eqnarray}
a_{1}&=&
\cos 2\varphi \left (\dfrac{\bar C^{1}_{23}-\bar C^{2}_{31}}{2}\right )+\sin 2\varphi\left (\bar C^{1}_{31}\right )
\\
a_{2}&=& -\sin 2\varphi \left (\dfrac{\bar C^{1}_{23}-\bar C^{2}_{31}}{2}\right )+\cos 2\varphi \left (\bar C^{1}_{31}\right )
\end{eqnarray}

\end{proposition}
\begin{proof}
If a coframe field $\eta = \{\eta^a\}$ is a local section of $B_2 \to M$, then 
the equations  \eqref{eq:1_12} are written as follows:  
\begin{equation}
\label{eq:1_17}
\begin{cases}
\theta^{1}= \cos\varphi \bar\eta^{1} + \sin\varphi\bar\eta^{2}\\
\theta^{2}= -\sin\varphi \bar\eta^{1} + \cos\varphi\bar\eta^{2}\\
\theta^{3}= \bar\eta^{3}
\end{cases}
\iff
\begin{cases}
\bar\eta^{1}= \cos\varphi \theta^{1} - \sin\varphi\theta^{2}\\
\bar\eta^{2}= \sin\varphi \theta^{1} + \cos\varphi\theta^{2}\\
\bar\eta^{3}= \theta^{3}
\end{cases}
\end{equation}
Then
\begin{equation}
\label{eq:1_18}
\begin{cases}
\theta^{2}\wedge\theta^{3}=\sin\varphi\bar\eta^{3}\wedge\bar\eta^{1}+\cos\varphi\bar\eta^{2}\wedge\bar\eta^{3}\\
\theta^{3}\wedge\theta^{1}=\cos\varphi\bar\eta^{3}\wedge\bar\eta^{1}-\sin\varphi\bar\eta^{2}\wedge\bar\eta^{3}\\
\theta^1 \wedge \theta^2 = \bar\eta^1 \wedge \bar\eta^2 
\end{cases}
\end{equation}
\begin{equation}\label{eq:1_19}
\begin{cases}
d\bar\eta^{1}=-\sin\varphi d\varphi\wedge\theta^{1} + \cos\varphi d\theta^{1}-\cos\varphi d\varphi \wedge\theta^{2} - \sin\varphi d\theta^{2}\\
d\bar\eta^{2}=\cos\varphi d\varphi\wedge\theta^{1} + \sin\varphi d\theta^{1}-\sin\varphi d\varphi \wedge\theta^{2} + \cos\varphi d\theta^{2}\\
d\bar\eta^{3}=d\theta^{3}
\end{cases}
\end{equation}
To \eqref{eq:1_19} we substitute $d\theta^a$ from  \eqref{eq:1_15} and then $\theta^a$ from \eqref{eq:1_17} and use \eqref{eq:1_18}, finally we arrive at  
\begin{equation}
\label{eq:1_20}
\begin{cases}
d\bar\eta^{1}=&(\alpha-d\varphi) \wedge \bar\eta^{2}+
 \left (a_{1}\cos 2\varphi - a_{2}\sin 2\varphi\right )\bar\eta^{2}\wedge\bar\eta^{3}+
\\
&\left (a_{1}\sin 2\varphi + 
a_{2}\cos 2\varphi\right )\bar\eta^{3}\wedge\bar\eta^{1}
\\
d\bar\eta^{2}=&(d\varphi - \alpha) \wedge\bar\eta^{1} +
\left (a_{1}\sin 2\varphi + a_{2}\cos 2\varphi\right )\bar\eta^{2}\wedge\bar\eta^{3} +
\\
& \left (-a_{1}\cos 2\varphi + a_{2}\sin 2\varphi\right )\bar\eta^{3}\wedge\bar\eta^{1}
\\
d\bar\eta^{3}=&\bar\eta^{1}\wedge\bar\eta^{2}
\end{cases}
\end{equation}
Let us set $\alpha - d\varphi = p_1 \bar\eta^1 + p_2 \bar\eta^2 + p_3 \bar\eta^3 + p_4 d\varphi$ and substitute to \eqref{eq:1_20}. We get
\begin{equation}
\label{eq:1_21}
\begin{cases}
d\bar\eta^{1}=&
\left(a_{1}\cos 2\varphi - a_{2}\sin 2\varphi - p_3 \right)\bar\eta^{2}\wedge\bar\eta^{3}+
\left (a_{1}\sin 2\varphi + a_{2}\cos 2\varphi\right )\bar\eta^{3}\wedge\bar\eta^{1} + 
\\
& p_1 \bar\eta^1 \wedge \bar\eta^2 + p_4 d\varphi \wedge \bar\eta^2
\\
d\bar\eta^{2} =& \left (a_{1}\sin 2\varphi + a_{2}\cos 2\varphi\right )\bar\eta^{2}\wedge\bar\eta^{3} +
\left (-a_{1}\cos 2\varphi + a_{2}\sin 2\varphi - 
p_3 \right)\bar\eta^{3}\wedge\bar\eta^{1} + 
\\
&p_2 \bar\eta^1 \wedge \eta^2 + p_4 \bar\eta^1 \wedge d\varphi
\\
d\bar\eta^{3}=&\bar\eta^{1}\wedge\bar\eta^{2}
\end{cases}
\end{equation}
At the same time, $d\bar{\eta}^a = \bar{C}^a_{bc}  \bar{\eta}^b \wedge \bar{\eta}^c$, hence we get that $p_4 =0$ and $\bar C^a_{bc}$ are expressed as follows:
\begin{eqnarray}
&&
\bar C^1_{23}  = a_{1}\cos 2\varphi - a_{2}\sin 2\varphi - p_3 
\label{eq:1_22_1}
\\
&&
\bar C^1_{31}  = a_{1}\sin 2\varphi + a_{2}\cos 2\varphi
\label{eq:1_22_2}
\\
&&
\bar C^1_{12}  = p_1 
\label{eq:1_22_3}
\\
&&
\bar C^2_{23}  = a_{1}\sin 2\varphi + a_{2}\cos 2\varphi
\label{eq:1_22_4}
\\
&&
\bar C^2_{31}  = -a_{1}\cos 2\varphi + a_{2}\sin 2\varphi - p_3
\label{eq:1_22_5}
\\
&&
\bar C^2_{12}  = p_2 
\label{eq:1_22_6}
\\
&&
\bar C^3_{23}  = 0 
\label{eq:1_22_7}
\\
&&
\bar C^3_{31}  = 0 
\label{eq:1_22_8}
\\
&&
\bar C^3_{12}  = 1 
\label{eq:1_22_9}
\label{eq:1_22}
\end{eqnarray}

From \eqref{eq:1_22_7}--\eqref{eq:1_22_9} we get claim a).

The equations \eqref{eq:1_22_3}, \eqref{eq:1_22_6}, and the sum of equations \eqref{eq:1_22_1}, \eqref{eq:1_22_5} give us $p_1$, $p_2$, and $p_3 = -\frac{1}{2}\left (\bar C^{1}_{23}+\bar C^{2}_{31}\right)$, thus we prove claim b).

The equations \eqref{eq:1_22_1}, \eqref{eq:1_22_2} with $p_3$ substituted give us claim c),

Finally note that \eqref{eq:1_22_2}, \eqref{eq:1_22_4} imply that $\bar C^1_{31} = \bar C^2_{23}$, this also can be proved in the following way. We have $d\bar\eta^3 = \bar\eta^1 \wedge \bar\eta^2$. From this it follows that $0 = d\bar\eta^1 \wedge \bar\eta^2 - \bar\eta^1 \wedge d\bar\eta^2$, hence we obtain $\bar C^1_{31} = \bar C^2_{23}$.
\end{proof}

\begin{corollary}
\label{cor:1}
The $SO(2)$-structure $B_2 \to M$, where $B_2$ is a $SO(2)$-principal subbundle of $B_0$ such that the tautological forms $\theta^a$ on $B_2$ satisfy structure equations~\eqref{eq:1_15} is unique. 
\end{corollary}

\begin{proof}
Let $B_2$ and $\hat B_2$ be $SO(2)$-principal subbundles of $B_0$ such that the tautological forms satisfy structure equations~\eqref{eq:1_15}. Take local sections  $\eta^a$ and $\hat\eta^a$ of $B_2$ and $\hat B_2$, respectively. 

By Proposition~\ref{prop:2} a), we have  
\begin{equation}
d\eta^3 = \eta^1 \wedge \eta^2 \text{ and } d\hat{\eta}^3 = \hat{\eta}^1 \wedge \hat{\eta}^2
\label{eq:200_1}
\end{equation} 

Let $\Omega$ be the area form on $\Delta$ determined by the metric $\langle \cdot,\cdot \rangle$.
Since $\eta^a$ and $\hat\eta^a$ are sections of $B_0$, we have 
$$
\eta^1 \wedge \eta^2 |_\Delta = \Omega = \hat\eta^1 \wedge \hat\eta^2 |_\Delta
$$
By the same reason, 
we have $\hat\eta^3 = e^f \eta^3$, hence $d\hat\eta^3 = e^f df \wedge \eta^3 +  e^f d\eta^3$.  We restrict it to $\Delta$ and from $d\hat\eta^3 |_\Delta = d\eta^3 |_\Delta$ get that $e^f = 1$. Hence $\hat\eta^3 = \eta^3$.
Now 
\begin{align}
\hat\eta^1 &= \cos\varphi \eta^1 - \sin\varphi \eta^2 + a \eta^3 
\\
\hat\eta^2 &= \sin\varphi \eta^1 + \cos\varphi \eta^2 + b \eta^3 
\label{eq:3}
\end{align}
But 
$$
\hat\eta^1 \wedge \hat \eta^2 = d\hat\eta^3 = d\eta^3 = \eta^1 \wedge \eta^2,
$$
so one can easily prove that $a=b=0$.
\end{proof}

\begin{corollary}
The function 
\begin{equation}
\label{Eq:formula-M}
\bar{\mathcal{M}}=\left (a_{1}\right )^{2}+\left (a_{2}\right )^{2}=\left (\dfrac{\bar C^{1}_{23}-\bar C^{2}_{31}}{2}\right )^{2}+\left (\bar C^{1}_{31}\right )^{2}
\end{equation}
is a pullback of a function $\mathcal{M} : M \to \mathbb{R}$, i.\,e. $\bar{\mathcal{M}} = \mathcal{M} \circ \pi$, where  
\begin{equation}
 \mathcal{M}=\left (\dfrac{C^{1}_{23}-C^{2}_{31}}{2}\right )^{2}+\left (C^{1}_{31}\right )^{2}.
\end{equation}
\end{corollary}

From Corollary~\ref{cor:1} it follows that the $SO(2)$-structure $B_2$ is uniquely determined by the sub-Riemannian surface $\mathcal{S}$.  
As the coframe field $\left\{\theta^{1}, \theta^{2}, \theta^{3}, \alpha\right \}$ on $B_2$ is uniquely determined, we see that the functions $a_1, a_2 : B_2 \to \mathbb{R}$ as well as the form $\alpha$ are uniquely determined by $\mathcal S$. Thus we get  
\begin{corollary} 
The function $\mathcal{M}$ 
is an invariant of the sub-Riemannian surface $\mathcal{S}$.
\end{corollary}

\subsubsection{Curvature of a contact sub-Riemannian surface}

Let us write down \eqref{eq:1_15} as follows: 
\begin{eqnarray}
&&
d\theta^{1}=\alpha\wedge\theta^{2}+a_{1}\theta^{2}\wedge\theta^{3}+a_{2}\theta^{3}\wedge\theta^{1}
\label{eq:1_100_1}
\\
&&
d\theta^{2}=-\alpha\wedge\theta^{1}+a_{2}\theta^{2}\wedge\theta^{3}-a_{1}\theta^{3}\wedge\theta^{1}
\label{eq:1_100_2}
\\
&&
d\theta^{3}=\theta^{1}\wedge\theta^{2}
\label{eq:1_100_3}
\end{eqnarray}
Take the exterior differential of \eqref{eq:1_100_1}, then we get
\begin{equation}
\begin{split}
 0 = d\alpha \wedge \theta^2 - \alpha \wedge d\theta^2 + da_1 \wedge \theta^2 \wedge \theta^3 + a_1 d\theta^2 \wedge \theta^3 -  a_1 \theta^2 \wedge d\theta^3 + \\
da_2 \wedge \theta^3 \wedge \theta^1 + a_2 d\theta^3 \wedge \theta^1 -  a_2 \theta^3 \wedge d\theta^1
\end{split}
\label{eq:1_101}
\end{equation}
To \eqref{eq:1_101} we substitute \eqref{eq:1_100_1}--\eqref{eq:1_100_3} and get 
\begin{equation}
d\alpha \wedge \theta^2 + 2 a_1 \alpha \wedge \theta^3 \wedge \theta^1 - 2 a_2 \alpha \wedge \theta^2 \wedge \theta^3 + da_1 \wedge \theta^2 \wedge \theta^3 + da_2 \wedge \theta^3 \wedge \theta^1 = 0.
\label{eq:1_102}
\end{equation}
In the same manner from \eqref{eq:1_100_2} we get 
\begin{equation}
- d\alpha \wedge \theta^1 + 2 a_1 \alpha \wedge \theta^2 \wedge \theta^3 + 2 a_2 \alpha \wedge \theta^3 \wedge \theta^1 + da_2 \wedge \theta^2 \wedge \theta^3 - da_1 \wedge \theta^3 \wedge \theta^1= 0.
\label{eq:1_103}
\end{equation}
Now we consider the expansions: 
\begin{equation}
\begin{cases}
da_1 = a \alpha + A_1 \theta^1 + A_2 \theta^2 + A_3 \theta^3
\\
da_2 = b \alpha + B_1 \theta^1 + B_2 \theta^2 + B_3 \theta^3
\\
d\alpha = P_1 \alpha \wedge \theta^1 +  P_2 \alpha \wedge \theta^2 + P_3 \alpha \wedge \theta^3 
+ X_{23} \theta^2 \wedge \theta^3 + X_{31} \theta^3 \wedge \theta^1 + X_{12} \theta^1 \wedge \theta^2.
\end{cases}
\label{eq:1_104}
\end{equation}
We will express $P_a$ and $X_{ab}$ in terms of $A_a$ and $B_b$. To do it we substitute \eqref{eq:1_104} to \eqref{eq:1_102} and get 
\begin{equation}
\begin{split}
P_1 \alpha \wedge \theta^1 \wedge \theta^2 + P_3 \alpha \wedge \theta^3 \wedge \theta^2 + X_{31} \theta^3 \wedge \theta^1 \wedge \theta^2 + 2 a_1 \alpha \wedge \theta^3 \wedge \theta^1 - 
\\
2 a_2  \alpha \wedge \theta^2 \wedge \theta^3 + 
 a \alpha \wedge \theta^2 \wedge \theta^3 + A_1 \theta^1 \wedge \theta^2 \wedge \theta^3 + 
\\ 
 b \alpha \wedge \theta^3 \wedge \theta^1 + B_2 \theta^2 \wedge \theta^3 \wedge \theta^1 = 0.
\end{split}
\label{eq:1_105}
\end{equation} 
From this we get 
\begin{equation}
P_1 = 0, \quad P_3 - a + 2a_2 = 0, \quad X_{31} + A_1 + B_2 = 0, \quad 2a_1 + b = 0.
\label{eq:1_106}
\end{equation}
In the same manner, substituting \eqref{eq:1_104} to \eqref{eq:1_103}, we get 
\begin{equation}
P_2 = 0, \quad -P_3 - a + 2a_2 = 0, \quad -X_{23} + B_1 - A_2 = 0, \quad 2a_1 + b = 0.
\label{eq:1_107}
\end{equation}
From \eqref{eq:1_106} and \eqref{eq:1_107} we get
\begin{equation}
\begin{split}
P_1 = P_2 = P_3 = 0, \quad a = 2a_2, \quad b = -2 a_1, 
\\
 X_{31} = -A_1 - B_2, \quad X_{23} = B_1 - A_2.
\label{eq:1_108}
\end{split}
\end{equation}
Thus  only $X_{12}$ is undetermined, and we denote it by $\bar{\mathcal{K}}$. 
In this way we obtain 
\begin{equation}
\begin{cases}
da_{1}=2a_{2}\alpha + A_{1}\theta^{1}+A_{2}\theta^{2}+A_{3}\theta^{3}\\
da_{2}=-2a_{1}\alpha + B_{1}\theta^{1}+B_{2}\theta^{2}+B_{3}\theta^{3}\\
d\alpha = \bar{\mathcal{K}}\theta^{1}\wedge\theta^{2}+\left (B_{1}-A_{2}\right )\theta^{2}\wedge\theta^{3}+\left (-A_{1}-B_{2}\right )\theta^{3}\wedge\theta^{1}
\end{cases}
\label{eq:1_108_1}
\end{equation}
Now let us express $\bar{\mathcal{K}}$ in terms of $\bar C^c_{ab}$.
To do it, we use \eqref{eq:1_17} and \eqref{eq:1_18}. 
Then, from \eqref{eq:1_90} we get 
\begin{equation}
\begin{split}
d\alpha = d\bar C^1_{12} \wedge \bar\eta^1 + \bar C^1_{12} \wedge d\bar \eta^1 + d\bar C^2_{12} \wedge \bar \eta^2 + \bar C^2_{12} \wedge d\bar \eta^2 - 
\\
 \frac{1}{2}d(\bar C^1_{23} + \bar C^2_{31})\wedge \bar \eta^3 - \frac{1}{2}(\bar C^1_{23} + \bar C^2_{31})\wedge d\bar \eta^3. 
\end{split}
\label{eq:1_110}
\end{equation}
Let us take the frame field $\{\bar E_1,\bar E_2,\bar E_3,\bar E_4\}$ dual to $\{\bar\eta^1,\bar\eta^2,\bar\eta^3,\alpha\}$.  Note that $\bar C^c_{ab}$ depend only on the base coordinates, so $E_4 \bar C^c_{ab} = 0$. 
Therefore,
\begin{equation}
\begin{split}
d \bar C^1_{23} = \bar E_1 \bar C^1_{23} \bar\eta^1 + \bar E_2 \bar C^1_{23} \bar\eta^2 + \bar E_3 \bar C^1_{23} \bar\eta^3;
\\
d \bar C^2_{23} = \bar E_1 \bar C^2_{23} \bar\eta^1 + \bar E_2 \bar C^2_{23} \bar\eta^2 + \bar E_3 \bar C^2_{23} \bar\eta^3.
\label{eq:1_111}
\end{split}
\end{equation}
If we substitute \eqref{eq:1_111} to \eqref{eq:1_110} we obtain expansion  
\begin{multline}
d\alpha = (\bar E_{1}\bar C^{2}_{12}-\bar E_{2}\bar C^{1}_{12}+\left (\bar C^{1}_{12}\right )^{2}+\left (\bar C^{2}_{12}\right )^{2}-
\\
\frac{1}{2}(\bar C^{1}_{23}+\bar C^{2}_{31})) \bar\eta^1 \wedge \bar\eta^2 + (\dots) \bar\eta^3 \wedge \bar\eta^1 + (\dots) \bar\eta^2 \wedge \bar\eta^3
\label{eq:1_112}
\end{multline}
where \dots stands for the coefficient we are not interested in now.
Then, use \eqref{eq:1_18}, and from  \eqref{eq:1_112} we get
\begin{multline}
d\alpha = (\bar E_{1}\bar C^{2}_{12}-\bar E_{2}\bar C^{1}_{12}+\left (\bar C^{1}_{12}\right )^{2}+\left (\bar C^{2}_{12}\right )^{2}-
\\
\frac{1}{2}(\bar C^{1}_{23}+\bar C^{2}_{31})) \theta^1 \wedge \theta^2 + (\dots) \theta^3 \wedge \theta^1 + (\dots) \theta^2 \wedge \theta^3
\label{eq:1_113}
\end{multline}
Compare \eqref{eq:1_108_1} and \eqref{eq:1_112}, then we finally find 
\begin{equation}
\label{Eq:formula-K1}
\bar{\mathcal{K}} = \bar E_{1}\bar C^{2}_{12}-\bar E_{2}\bar C^{1}_{12}+\left (\bar C^{1}_{12}\right )^{2}+\left (\bar C^{2}_{12}\right )^{2} -\frac{1}{2}(\bar C^{1}_{23}+\bar C^{2}_{31}). 
\end{equation}
It is clear that $\bar E_a$ are horizontal lifts of vector fields $E_a$ which constitute a local frame field on $U$, and $\bar E_a \bar C^b_{cd} = (E_a C^b_{cd}) \circ \pi$. As $\bar C^a_{bc} = C^a_{bc} \circ \pi$, we have $\bar{\mathcal{K}} = \mathcal{K} \circ \pi$, where
\begin{equation}
\label{Eq:formula-K}
\mathcal{K} = E_{1} C^{2}_{12}-  E_{2}  C^{1}_{12}+\left (  C^{1}_{12}\right )^{2}+\left (  C^{2}_{12}\right )^{2} -\frac{1}{2}(  C^{1}_{23}+  C^{2}_{31}). 
\end{equation}

The function $\mathcal{K}$ is called the \textit{curvature} of $\mathcal{S}$, and it is clear that $\mathcal{K}$ is an invariant of $\mathcal{S}$.  

We result our investigations of invariants of a contact sub-Riemannian surface in the following 
\begin{theorem}
\label{Th:Theorem-2}
Let $\mathcal S = (M,\Delta,\langle \cdot,\cdot \rangle)$ be a contact sub-Riemannian surface. Then, for any $p \in M$, in a neighborhood $U$ of $p$ a coframe field $\eta = \{\eta^a\}$ exists such that
\begin{equation}\label{Eq:structure-functions-coframe}
\left (
\begin{array}{c}
d\eta^{1}\\
d\eta^{2}\\
d\eta^{3}
\end{array}
\right )
=
\left (
\begin{array}{ccc}
C^{1}_{23} & C^{1}_{31} & C^{1}_{12}\\
C^{2}_{23} & C^{2}_{31} & C^{2}_{12}\\
0 & 0 & 1
\end{array}
\right )
\left (
\begin{array}{c}
\eta^{2}\wedge\eta^{3}\\
\eta^{3}\wedge\eta^{1}\\
\eta^{1}\wedge\eta^{2}\\
\end{array}
\right )
\end{equation}
The functions 
\begin{eqnarray}
&&
 \mathcal{M}=\left (\dfrac{C^{1}_{23}-C^{2}_{31}}{2}\right )^{2}+\left (C^{1}_{31}\right )^{2}
\\
&&
\mathcal{K} = E_{1} C^{2}_{12}-  E_{2}  C^{1}_{12}+\left (  C^{1}_{12}\right )^{2}+\left (  C^{2}_{12}\right )^{2} -\frac{1}{2}(  C^{1}_{23}+  C^{2}_{31}) 
\label{eq:invariants}
\end{eqnarray}
do not depend on the choice of coframe field $\eta$ with structure equations \eqref{Eq:structure-functions-coframe} and are correctly defined on $M$. 
\end{theorem}

\section{Symmetries of contact sub-Riemannian surfaces}

A \textit{symmetry} of a sub-Riemannian surface $\mathcal S = (M,\Delta,\langle \cdot,\cdot \rangle)$ is a local diffeomorphism $F : M \to M$ such that, for any $p \in M$, 
\begin{eqnarray}
&&
F(\Delta_p) = \Delta_{F(p)},
\label{eq:1_150_1}
\\
&&
\langle dF(W), dF(W) \rangle_{F(p)} = \langle W, W \rangle_{p}, \forall W \in \Delta_p. 
\label{eq:1_150_2}
\end{eqnarray}

A vector field $V \in \mathfrak{X}(M)$ is called an \textit{infinitesimal symmetry } if its flow consists of symmetries.

%%%%%%%%%%%%%%%%%%%%%%%%%-------Therema 3

\begin{theorem}
\label{Th:Theorem-3}
Let $\mathcal S = (M,\Delta,\langle \cdot,\cdot \rangle)$ be a contact sub-Riemannian surface. 
Let $\eta = \{\eta^a\}$ be a coframe field in a neighborhood $U$ of $p \in M$ such that \eqref{Eq:structure-functions-coframe} holds true $($$\eta$ exists by Theorem~\ref{Th:Theorem-2}$)$, and $\{E_a\}$ be the dual frame field. 

a) For any infinitesimal symmetry $V$ of $\mathcal{S}$, a unique  function $f : U \to \mathbb{R}$ exists such that   
\begin{equation}\label{Eq:form-of-function-f}
V=-E_{2}(f)E_{1} + E_{1}(f)E_{2} + fE_{3} \text{ and } E_3 f = 0.
\end{equation}

b) Let $\mathcal{M}$ and $\mathcal{K}$ be the invariants of $\mathcal{S}$ $($ see Theorem~\ref{Th:Theorem-2}$)$. Then, if $V$ is transversal to $\Delta$ and   
$E_{1}\mathcal{K}E_{2}\mathcal{M}-E_{2}\mathcal{K}E_{1}\mathcal{M} \ne 0$, the function $\ln f$ satisfies the following system of partial differential equations:  
\begin{equation}\label{Eq:system-to-find-f}
\begin{cases}
E_{1}(\ln f) = \dfrac{E_{3}\mathcal{K}E_{1}\mathcal{M}-E_{1}\mathcal{K}E_{3}\mathcal{M}}{E_{1}\mathcal{K}E_{2}\mathcal{M}-E_{2}\mathcal{K}E_{1}\mathcal{M}}\\
E_{2}(\ln f) = \dfrac{E_{3}\mathcal{K}E_{2}\mathcal{M}-E_{2}\mathcal{K}E_{3}\mathcal{M}}{E_{1}\mathcal{K}E_{2}\mathcal{M}-E_{2}\mathcal{K}E_{1}\mathcal{M}}\\
E_{3}(\ln f) = 0. 
\end{cases}
\end{equation}
\end{theorem}

%----------------
\begin{proof}
Let $V$ be an infinitesimal symmetry of $\mathcal{S}$, and $\phi_t$ be the flow of $V$. 
As $\{E_1(p), E_2(p)\}$ is an orthonormal frame of $\Delta (p)$, we have, by definition of infinitesimal symmetry \eqref{eq:1_150_1}, \eqref{eq:1_150_2}, that    
\begin{equation}
\begin{cases}
d\phi_{t}E_{1}(p)=\cos\varphi (t) E_{1}\left (\phi_{t}(p)\right )+\sin\varphi (t) E_{2}\left (\phi_{t}(p)\right )\\
d\phi_{t}E_{2}(p)=-\sin\varphi (t) E_{1}\left (\phi_{t}(p)\right )+\cos\varphi (t) E_{2}\left (\phi_{t}(p)\right )
\end{cases}
\end{equation}
Hence
\begin{equation}
\begin{cases}
d\phi_{t}E_{1}\left (\phi_{-t}(p)\right )=\cos\varphi (t) E_{1}(p)+\sin\varphi (t) E_{2}(p)\\
d\phi_{t}E_{2}\left (\phi_{-t}(p)\right )=-\sin\varphi (t) E_{1}(p)+\cos\varphi (t) E_{2}(p)
\end{cases}
\end{equation}
From this follows that
\begin{equation}\label{Eq:lambda-system}
\begin{cases}
[E_{1}, V] = \lambda E_{2}\\
[E_{2}, V] = -\lambda E_{1}\\
\end{cases}
\end{equation}
because
\begin{equation}
\mathcal{L}_{V}E_{1}=[V,E_{1}]=\left.\dfrac{d}{dt}\right|_{t=0} d\phi_t E_{1}\left (\phi_{-t}(p)\right )=\varphi' (0)E_{2}=-\lambda E_{2}
\end{equation}
and similar for $[V,E_{2}]=\lambda E_{1}$. 

On the other hand we know that the structure equations for the dual frame $E=\left (E_{1}, E_{2}, E_{3}\right )$ are
\begin{equation}
\begin{cases}
[E_{1}, E_{2}]=c^{1}_{12}E_{1}+c^{2}_{12}E_{2}+c^{3}_{12}E_{3}\\
[E_{3}, E_{1}]=c^{1}_{31}E_{1}+c^{2}_{31}E_{2}+c^{3}_{31}E_{3}\\
[E_{2}, E_{3}]=c^{1}_{23}E_{1}+c^{2}_{23}E_{2}+c^{3}_{23}E_{3}\\
\end{cases}
\end{equation}
But $c^{i}_{jk}=-C^{i}_{jk}$ from \eqref{Eq:structure-functions-coframe}. Therefore
\begin{equation}
\begin{cases}
[E_{1}, E_{2}]=-\left (C^{1}_{12}E_{1}+C^{2}_{12}E_{2}+E_{3}\right )\\
[E_{3}, E_{1}]=-\left (C^{1}_{31}E_{1}+C^{2}_{31}E_{2}\right )\\
[E_{2}, E_{3}]=-\left (C^{1}_{23}E_{1}+C^{2}_{23}E_{2}\right )\\
\end{cases}
\label{eq:1_140}
\end{equation}
Substituting $V = V^1 E_1 + V^2 E_2 + V^3 E_3$ to the first equation in \eqref{Eq:lambda-system}, we get 
\begin{multline}
\lambda E_2 = [E_1, V^1 E_1 + V^2 E_2 + V^3 E_3] = 
\\
E_1V^1\, E_1 + E_1 V^2\, E_2 + V^2 [E_1,E_2] + E_1V^3\,E_3 + V^3 [E_1,E_3] = \\
(E_1V^1 - V^2 C^1_{12} + V^3 C^1_{31}) E_1 + (E_1V^2 - V^2 C^2_{12} + V^3 C^2_{31}) E_2 + (E_1 V^3 - V^2)E_3.  
\label{eq:1_151}
\end{multline}
In the same manner, substituting $V = V^1 E_1 + V^2 E_2 + V^3 E_3$ to the second equation in \eqref{Eq:lambda-system}, we get 
\begin{multline}
-\lambda E_1 = [E_2, V^1 E_1 + V^2 E_2 + V^3 E_3] = 
\\
E_2 V^1\, E_1 + V^1 [E_2,E_1] +  E_2 V^2\, E_2 +  E_2 V^3\,E_3 + V^3 [E_2,E_3] = 
\\
(E_2V^1 + V^1 C^1_{12} - V^3 C^1_{23}) E_1 + (E_2 V^2 + V^1 C^2_{12} - V^3 C^2_{23}) E_2 + (E_2 V^3 + V^1)E_3.  
\label{eq:1_152}
\end{multline}
From \eqref{eq:1_151} and \eqref{eq:1_152} we obtain the following equation system:
\begin{eqnarray}
&&
E_1V^1 - V^2 C^1_{12} + V^3 C^1_{31} = 0,
\label{eq:1_153_1}
\\
&&
E_1V^2 - V^2 C^2_{12} + V^3 C^2_{31} = \lambda,
\label{eq:1_153_2}
\\
&&
E_1 V^3 - V^2 = 0,
\label{eq:1_153_3}
\\
&&
E_2V^1 + V^1 C^1_{12} - V^3 C^1_{23} = -\lambda,
\label{eq:1_153_4}
\\
&&
E_2 V^2 + V^1 C^2_{12} - V^3 C^2_{23} = 0,
\label{eq:1_153_5}
\\
&&
E_2 V^3 + V^1 = 0.
\label{eq:1_153_6}
\end{eqnarray}
Let us set $f = \eta^3(V) = V^3$, then \eqref{eq:1_153_3} and \eqref{eq:1_153_6} give 
\begin{equation}
V=-E_{2}(f)E_{1} + E_{1}(f)E_{2} + fE_{3}. 
\end{equation}
Now substitute $V^1 = -E_2 f$, $V^2 = E_1 f$, and $V^3 = f$ to \eqref{eq:1_153_1} and \eqref{eq:1_153_5}:
\begin{eqnarray*}
- E_1 E_2 f - C^1_{12} E_1 f + C^1_{31} f = 0,
\\
E_2 E_1 f  - C^2_{12} E_2 f -  C^2_{23} f = 0,
\end{eqnarray*}
Summing these equalities, we arrive at 
\begin{equation}
[E_2,E_1]f - C^1_{12} E_1 f - C^2_{12} E_2 f = 0
\end{equation}
but, by  the first equation in \eqref{eq:1_140}, this means that $E_3 f = 0$.
Thus we have proved \eqref{Eq:form-of-function-f} and claim a).

Let us now prove b). As $V$ is an infinitesimal symmetry of the sub-Riemannian surface $\mathcal{S}$, we have  
\begin{equation}\label{Eq:system-VK-VM}
\begin{cases}
V\mathcal{K}=0\\
V\mathcal{M}=0
\end{cases}
\end{equation}
If we substitute \eqref{Eq:form-of-function-f} to \eqref{Eq:system-VK-VM}, we get 
\begin{equation}
\begin{cases}
- E_{1}\mathcal{K} E_{2}f + E_{2}\mathcal{K} E_{1}f + E_{3}\mathcal{K} f = 0
\\
- E_{1}\mathcal{M} E_{2}f + E_{2}\mathcal{M} E_{1}f + E_{3}\mathcal{M} f = 0
\end{cases}
\end{equation}
Since $V$ is transversal to $\Delta$, and hence $f$ does not vanish in $U$,  we can divide both equations  by $f$ and  obtain the system of linear equations in $E_1 \ln f$ and $E_2 \ln f$:
\begin{equation}
\begin{cases}
- E_{1}\mathcal{K}\, E_{2}\ln f + E_{2}\mathcal{K}\,  E_{1} \ln f + E_{3}\mathcal{K}  = 0
\\
- E_{1}\mathcal{M}\,  E_{2} \ln f + E_{2}\mathcal{M}\,  E_{1} \ln f + E_{3}\mathcal{M}  = 0
\end{cases}
\label{eq:1_157}
\end{equation}
If $E_{1}\mathcal{K}E_{2}\mathcal{M}-E_{2}\mathcal{K}E_{1}\mathcal{M} \ne 0$, this system has the unique solution 
\begin{eqnarray}
E_{1}(\ln f) &=& \dfrac{E_{3}\mathcal{K}E_{1}\mathcal{M}-E_{1}\mathcal{K}E_{3}\mathcal{M}}{E_{1}\mathcal{K}E_{2}\mathcal{M}-E_{2}\mathcal{K}E_{1}\mathcal{M}}
\label{eq:1_155_1}
\\
E_{2}(\ln f) &=& \dfrac{E_{3}\mathcal{K}E_{2}\mathcal{M}-E_{2}\mathcal{K}E_{3}\mathcal{M}}{E_{1}\mathcal{K}E_{2}\mathcal{M}-E_{2}\mathcal{K}E_{1}\mathcal{M}}
\label{eq:1_155_2}
\end{eqnarray}
To \eqref{eq:1_155_1} and \eqref{eq:1_155_2} we add $E_3 \ln f = 0$, which follows from \eqref{Eq:form-of-function-f}, and get the system \eqref{Eq:system-to-find-f}. Thus we have proved b).  
\end{proof}

%%%%%%%%%%%%%%%%%%%%%%%%%%%%%%%

\begin{remark}
If an infinitesimal symmetry $V$ of $\mathcal{S}$ lies in $\Delta$ at each point of an open set $W$, then $f$ is zero, and, by \eqref{Eq:form-of-function-f}, $V$ is zero, too. So, nonvanishing $V$ should be transversal to $\Delta$ almost everywhere.  
\end{remark}

\begin{remark}
Theorem~\ref{Th:Theorem-3} can be used in order to prove that a sub-Riemannian surface $\mathcal{S}$ does not admit nontrivial infinitesimal symmetries. To do it, it is sufficient to prove that the integrability conditions do not hold for the system  \eqref{Eq:system-to-find-f}. The integrability conditions have the form:
\begin{equation}
\begin{cases}
E_{1}(EQ2)-E_{2}(EQ1)=C^{1}_{12}EQ1+C^{2}_{12}EQ2\\
E_{3}EQ1=C^{1}_{31}EQ1+C^{2}_{31}EQ2\\
-E_{3}EQ2=C^{1}_{23}EQ1+C^{2}_{23}EQ2,
\end{cases}
\label{eq:1_156}
\end{equation}
where
\begin{eqnarray*}
EQ1  &=& \dfrac{E_{3}\mathcal{K}E_{1}\mathcal{M}-E_{1}\mathcal{K}E_{3}\mathcal{M}}{E_{1}\mathcal{K}E_{2}\mathcal{M}-E_{2}\mathcal{K}E_{1}\mathcal{M}},
\\
EQ2 &=& \dfrac{E_{3}\mathcal{K}E_{2}\mathcal{M}-E_{2}\mathcal{K}E_{3}\mathcal{M}}{E_{1}\mathcal{K}E_{2}\mathcal{M}-E_{2}\mathcal{K}E_{1}\mathcal{M}}.
\end{eqnarray*}
However, if the integrability conditions \eqref{eq:1_156} do hold for \eqref{Eq:system-to-find-f}, one can use this system in order to find $f$ and then $V$. In fact, we can take a natural frame field $\partial_a$ and write $E_{a}=B_{a}^{b}\partial_{b}$, Then  the equation system \eqref{Eq:system-to-find-f} can be rewritten as $\partial_{a} \ln f = g_{a}$, and the solution can be found by the well-known formula:
\begin{equation}
\ln f(x^{a})=\int_{0}^{1}x^{b}g_{b}(tx^{a})dt. 
\end{equation}
\end{remark}
\begin{remark}
The condition $\mathcal{D} = E_{1}\mathcal{K}E_{2}\mathcal{M}-E_{2}\mathcal{K}E_{1}\mathcal{M} \ne 0$, in general, does not hold. If $\mathcal{D} = 0$, the system \eqref{eq:1_157} may not have any solutions and this means that $\mathcal{S}$ does not admit any infinitesimal symmetries; or it may have infinitely many solutions, then we simply get an additional relation for $f$, which can be used in order to find infinitesimal symmetries by another method.  
\end{remark}

%%%%%%%%%%%%%%%%%%%%%%%%%%%%%%%%%%%

\subsection{Examples of infinitesimal symmetries}

\subsubsection{Heisenberg distribution}

Consider the Heisenberg distribution $\Delta$ given with respect to the standard coordinates  in $\mathbf{R}^{3}$ by the $1$-form
\begin{equation}\nonumber
\eta^{3}=dz+ydx-xdy.
\end{equation}
For the metric on $\Delta$ we take the metric induced from $\mathbb{R}^3$. By calculations, we get the following results: 

\begin{enumerate}
\item
The $SO(2)$-structure $B_2 \to \mathbb{R}^3$ is given by the coframe field 
\begin{equation}\nonumber
\begin{cases}
\eta^{1}=\dfrac{\left (2+3y^{2}\right )dx}{2\sqrt{1+y^{2}}}-\dfrac{3xydy}{2\sqrt{1+y^{2}}} +\dfrac{ydz}{2\sqrt{1+y^{2}}}\\
\eta^{2}=-\dfrac{xydx}{2\sqrt{1+y^{2}}\sqrt{1+x^{2}+y^{2}}}+\dfrac{\left (2+3x^{2}+2y^{2}\right )dy}{2\sqrt{1+y^{2}}\sqrt{1+x^{2}+y^{2}}} -\dfrac{xdz}{2\sqrt{1+y^{2}}\sqrt{1+x^{2}+y^{2}}}\\
\eta^{3}=-\dfrac{y}{2}\sqrt{1+x^{2}+y^{2}} dx+\dfrac{x}{2}\sqrt{1+x^{2}+y^{2}} dy-\dfrac{1}{2}\sqrt{1+x^{2}+y^{2}} dz
\end{cases}
\end{equation}
\item%
The orthonormal dual frame is 
\begin{equation}\nonumber
\begin{cases}
E_{1}=
\dfrac{1}{\sqrt{1+y^{2}}}\dfrac{\partial }{\partial x}-
\dfrac{y}{\sqrt{1+y^{2}}}\dfrac{\partial }{\partial z}\\
E_{2}=
\dfrac{xy}{\sqrt{1+y^{2}}\sqrt{1+x^{2}+y^{2}}}\dfrac{\partial }{\partial x}+
\dfrac{\sqrt{1+y^{2}}}{\sqrt{1+x^{2}+y^{2}}}\dfrac{\partial }{\partial y}+
\dfrac{x}{\sqrt{1+y^{2}}\sqrt{1+x^{2}+y^{2}}}\dfrac{\partial }{\partial z}\\
E_{3}=
\dfrac{y}{\left (1+x^{2}+y^{2}\right )^{3/2}}\dfrac{\partial }{\partial x}- 
\dfrac{x}{\left (1+x^{2}+y^{2}\right )^{3/2}}\dfrac{\partial }{\partial y}-
\dfrac{2+3x^{2}+3y^{2}}{\left (1+x^{2}+y^{2}\right )^{3/2}}\dfrac{\partial }{\partial z}
\end{cases}
\end{equation}
\item%
The structure functions \footnote{$d\eta^{i}=C^{i}_{jk}\eta^{j}\wedge\eta^{k}$, $[E_{j}, E_{k}]=c^{i}_{jk}E_{i}$, $C^{i}_{jk}=-c^{i}_{jk}$} $C^{i}_{jk}$ are
\begin{equation}\nonumber
\begin{array}{lll}
C^{1}_{23}= -\dfrac{1-3y^{2}}{\left (1+y^{2}\right )\left (1+x^{2}+y^{2}\right )} & C^{2}_{23} = -\dfrac{3xy}{\left (1+y^{2}\right )\left (1+x^{2}+y^{2}\right )^{3/2}} & C^{3}_{23} =0\\ 
C^{1}_{31} =-\dfrac{3xy}{\left (1+y^{2}\right )\left (1+x^{2}+y^{2}\right )^{3/2}} & C^{2}_{31}=-\dfrac{1-2x^{2}+y^{2}}{\left (1+y^{2}\right )\left (1+x^{2}+y^{2}\right )} & C^{3}_{31} =0\\ 
C^{1}_{12}=-\dfrac{3y}{\sqrt{1+y^{2}}\sqrt{1+x^{2}+y^{2}}} & C^{2}_{12}= \dfrac{2x}{\sqrt{1+y^{2}}\left (1+x^{2}+y^{2}\right )} & C^{3}_{12}=1\\ 
\end{array}
\end{equation}
\item %
The invariants
\begin{equation}\nonumber
\begin{cases}
\mathcal{M}=\dfrac{9}{4}\dfrac{\left (x^{2}+y^{2}\right )^{2}}{\left (1+x^{2}+y^{2}\right )^{4}}\\
\mathcal{K}=\dfrac{3\left (1+2x^{2}+4y^{2}+3x^{2}y^{2}+3y^{4}\right )}
{\left (1+y^{2}\right )\left (1+x^{2}+y^{2}\right )^{2}}
\end{cases}
\end{equation}
\item%
The family of functions $f$ which define symmetries
\begin{equation}\nonumber
f=A\sqrt{1+x^{2}+y^{2}}, \; \mbox{where}\; A=const.
\end{equation}

\item%
The connection form
\begin{eqnarray}
\alpha &=&
d\varphi -
\dfrac{1}{2\left (1+y^{2}\right )\left (1+x^{2}+y^{2}\right )^{3/2}}
\{
y(4+9x^{2}+16y^{2}+12x^{2}y^{2}+12y^{4})dx\nonumber\\
&-&x(2+7x^{2}+14y^{2}+12x^{2}y^{2}+12y^{4})dy
+(-2+3x^{2}+4y^{2}+6x^{2}y^{2}+6y^{4})dz
\}\nonumber
\end{eqnarray}
\end{enumerate}

%%%%%%%%%%%%%%%%%%%%%%%%%%%%%%%%%%%

\subsection{Cartan distribution}

Consider the Cartan distribution $\Delta$ given with respect to the standard coordinates in $\mathbf{R}^{3}$ by the $1$-form
\begin{equation}\nonumber
\eta^{3}=dz+ydx.
\end{equation}
\begin{enumerate}
\item
The $SO(2)$-structure $B_2 \to \mathbb{R}^3$ is given by the coframe field 
\begin{equation}\nonumber
\begin{cases}
\eta^{1}=\dfrac{1+2y^{2}}{\sqrt{1+y^{2}}}dx+\dfrac{y}{\sqrt{1+y^{2}}}dz\\
\eta^{2}=dy\\
\eta^{3}=-\sqrt{1+y^{2}}dx-\sqrt{1+y^{2}}dz
\end{cases}
\end{equation}
\item%
The orthonormal dual frame is
\begin{equation}\nonumber
\begin{cases}
E_{1}=
\dfrac{1}{\sqrt{1+y^{2}}}\dfrac{\partial }{\partial x} - \dfrac{y}{\sqrt{1+y^{2}}}\dfrac{\partial }{\partial z} \\
E_{2}=\dfrac{\partial }{\partial y} 
\\
E_{3}=
\dfrac{y}{\left (1+y^{2}\right )^{3/2}} - \dfrac{1+2y^{2}}{\left (1+y^{2}\right )^{3/2}}\dfrac{\partial }{\partial z} 
\end{cases}
\end{equation}
\item%
The structure functions are,
\begin{equation}\nonumber
\begin{array}{lll}
C^{1}_{23}=  -\dfrac{1-y^{2}}{\left (1+y^{2}\right )^{2}} & C^{2}_{23} = 0 & C^{3}_{23} = 0 \\ 
C^{1}_{31} = 0 & C^{2}_{31}= 0 & C^{3}_{31} =0 \\ 
C^{1}_{12}=-\dfrac{2y}{1+y^{2}} & C^{2}_{12}= 0 & C^{3}_{12}=1\\ 
\end{array}
\end{equation}
\item %
The invariants
\begin{equation}\nonumber
\begin{cases}
\mathcal{M}=\dfrac{1}{4}\dfrac{\left (1-2y^{2}\right )^{2}}{\left (1+y^{2}\right )^{4}}\\
\mathcal{K}=\dfrac{1+4y^{2}}{\left (1+y^{2}\right )^{2}}
\end{cases}
\end{equation}
\item%
The family of functions $f$ which define symmetries
\begin{equation}\nonumber
f=f(y,z+xy)
\end{equation}
\item%
The connection form
\begin{equation}\nonumber
\alpha =
d\varphi -
\dfrac{y\left (1+6y^{2}\right )}{\left (1+y^{2}\right )^{3/2}}dx +
\dfrac{y\left (1-4y^{2}\right )}{\left (1+y^{2}\right )^{3/2}}dz
\end{equation}

\end{enumerate}

%%%%%%%%%%%%%%%%%%%%%%%%%%%%%%%%%%%

\section{Noncontact sub-Riemannian surfaces of stable type}

Let us consider a sub-Riemannian surface $\left (\Delta , \langle \cdot , \cdot \rangle\right )$.
Now we do not assume that $\Delta$ is contact, so we admit that the set 
\begin{equation}
\Sigma = \{p \in \mathbf{R}^3 \mid (\omega \wedge d\omega)_p = 0\},
\end{equation}
where $\Delta$ is the kernel of the 1-form $\omega$, 
is, in general, non-empty.
However, we assume that 
\textit{$\Sigma$ is a 2-dimensional submanifold in $\mathbf{R}^3$ and the distribution $\Delta$ is transversal to $\Sigma$}. 

\begin{remark}
\label{rem:2_1}
The surface $\Sigma$ does not depend on the choice of $\omega$.
\end{remark}

\begin{remark}
\label{rem:2_2}
Any stable germ of a Pfaffian equation on a $3$-dimensional manifold is equivalent either to the germ of the $1$-form $\omega_0 = dz + xdy$, or b)~$\omega_1 = dy + x^2 dz$, at the origin \cite{Zhitomirskii}. For a contact distribution $\Delta$, the germ of $\Delta$ at each point is equivalent to the germ of the distribution determined by $\omega_0$. If a distribution $\Delta$ satisfies our assumption, then, for any point $p in \mathbf{R}^3  \setminus \Sigma$, the    the germ of $\Delta$ at $p$ is equivalent to the germ of the distribution determined by $\omega_0$, and, for $p \in \Sigma$, the germ is equivalent to the germ of the distribution determined by $\omega_0$.
\end{remark}

\subsection{Nonholonomity function of sub-Riemannian surface}

For the  sub-Riemannian surface $\left (\Delta ,\langle \cdot , \cdot \rangle\right )$ let us take a non-vanishing section $\omega$ of the bundle $Ann(\Delta)$.  Then, in a neighborhood $U$ of each point $p \in M$ take an positively oriented orthonormal frame field $\{E_1, E_2\}$ of $\Delta$. Then we define the function $\lambda_U : U \to \mathbf{R}$, $\lambda_U(q) = \omega([E_1,E_2](q))$. One can easily check that $\lambda_U$ does not depend on a choice of the frame field $\{E_1, E_2\}$, therefore if $U \cap V \ne \emptyset$, $\lambda_U |_{U \cap V} = \lambda_V |_{U \cap V}$. Therefore, we have correctly defined function $\lambda_\omega$ on $M$ by setting $\lambda_\omega |_U = \lambda_U$.  This function will be called the \textit{nonholonomity function of sub-Riemannian surface}. Note that this function depends on the choice of form $\omega$ and on the metric on $\Delta$.

\begin{proposition}
\label{prop:2_2}
The nonholonomity function has the following properties: 
\begin{itemize}
\item[a)] 
$\lambda_{e^\varphi\omega} = e^\varphi \lambda_\omega$;
\item[b)] 
$\lambda(p) = 0$ if and only if $p \in \Sigma$;
\item[c)] 
$d\lambda_\omega |_p \ne 0$ for any $p \in \Sigma$.
\item[d)] $d\omega |_\Delta = -\frac{1}{2}\lambda_\omega \Omega$, where $\Omega$ is the area 2-form on $\Delta$ determined by the metric.
\end{itemize} 
\end{proposition}
\begin{proof}
a) is evident from the definition of nonholonomity function.

b) In a neighborhood $U$ of a point $p$ take an positively oriented orthonormal frame field $\{E_1, E_2\}$ of $\Delta$ and a vector field $E_3$ such that $\omega(E_3)=1$. Then, $\{E_1, E_2,E_3\}$ is a frame field on $U$. We have $d\omega(E_1,E_2) = - \frac{1}{2}\omega([E_1,E_2) = - \frac{1}{2}\lambda_\omega$, and $\omega(E_1)=\omega(E_2)=0$, from this follows  
\begin{equation*}
d\omega\wedge\omega(E_1,E_2,E_3) = -\frac{1}{6}\lambda_\omega.
\end{equation*}
This proves b).

c) By our assumptions, for any $p \in \Sigma$, with respect to a coordinate system, $\omega=e^\varphi \omega_1$, where $\omega_1 = dz + x^2 dy$. From a) it follows that $\lambda_\omega = e^\varphi\lambda_{\omega_1}$. Also, $dx$, $dy$, $\omega_1$ is a coframe, hence $dx \wedge dy \wedge \omega_1 (E_1,E_2,E_3) \ne 0$, then $dx \wedge dy (E_1,E_2) \ne 0$. Therefore, $\lambda_{\omega_1} = e^\psi x$, for a function $\psi$, and, hence, $\Sigma \cap U$ is given by the equation $x=0$, and $\lambda_\omega = e^{\varphi+\psi} x$. From this immediately follows the required statement.  

d) For a positively oriented orthonormal frame field $\{E_1, E_2\}$ of $\Delta$ we have $\Omega(E_1,E_2)=1$ and $d\omega(E_1,E_2)=-\frac{1}{2}\omega([E_1,E_2]) = -\frac{1}{2}\lambda_\omega$. This proves d). 
\end{proof}

\subsection{Characteristic vector field}
Let us denote by $Ann(\Delta)$ the vector subbundle in $T^* M$ of rank $1$ whose fiber at $p \in M$ consists of $1$-forms vanishing at $\Delta_p$. 

\begin{proposition}
\label{prop:2_1}
For each point $p \in M$ there exists a section $\omega$ of $Ann(\Delta)$ in a neighborhood $U(p)$ which admits a vector field $V$ on $U(p)$ such that $L_V \omega = 0$ and $\omega(V)=1$. For a given $\omega$ this vector field is unique. 
\end{proposition}
\begin{proof}
For $p \in \Sigma$ we can take $\omega = dz + x^2 dy$, for the other $p$ we can take $\omega=dz+xdy$ with respect to an appropriate coordinate system (see Remark \ref{rem:2_2}). In both cases, the vector field $V = \frac{\partial}{\partial z}$ has the required properties. 

Let us prove that, for a given $\omega$, the vector field $V$ such that $L_V \omega = 0$ and $\omega(V)=1$ is unique.

Let us take a frame field $E_1$, $E_2$, $E_3$ on $U(p)$ such that $\Delta$ is spanned by $E_1$ and $E_2$, and  $E_3 = V$. Now let $W$ be a vector field with the required properties, then $W = W^1 E_1 + W^2 E_2 + W^3 E_3$. Since $\omega(V) = \omega(W) = 1$, we have $W^3=1$. 
As $E_3$ is an infinitesimal symmetry of $\omega$, $E_3$ is an infinitesimal symmetry of $\Delta$, too, therefore the vector fields $[E_3,E_1]$, $[E_3,E_2]$ are tangent to $\Delta$. 
From this follows that the vector fields $W^1 [E_1,E_2]$, $W^2 [E_1,E_2]$ are tangent to $\Delta$, but $\omega([E_1,E_2])\ne 0$ on $U(p) \setminus \Omega$, therefore $W^1 = W^2 = 0$ on  $U(p) \setminus \Omega$, and $W^1$ and $W^2$ vanish on $U(p)$.
Thus $W=E_3=V$ and the uniqueness has been proved.
\end{proof}

If for $p \in M$, a nonvanishing form $\omega \in Ann(\Delta)$ on a neighborhood $U$ of $p$ for which there exists a vector field $V$ on such that $L_V \omega = 0$ and $\omega(V)=1$ will be called a \textit{special form at} $p$, and $V$ the \textit{characteristic vector field of} $\omega$. 
Note that, if $\omega$ is special at each point of an open set $U \subset M$, then on $U$ we have a unique vector field $V$ such that  $L_V \omega = 0$ and $\omega(V)=1$.

Let us consider the form $\widetilde\omega =e^\varphi \omega$. In general, $\widetilde\omega$ is not special. 

\begin{proposition}
\label{prop:2_3}
a) If $p \in M \setminus \Omega$, then any nonvanishing form $\omega \in Ann(\Delta)$ is special.

b) If $p \in \Sigma$ and $\omega$ is special at $p$, then $\widetilde\omega = e^\varphi \omega$ is special at $p$ if and only if  on a neighborhood $U$ of $p$ we have $d\varphi|_\Delta = \lambda_w \xi$, where $\xi$ is a nonvanishing $1$-form on $\Delta$ in $U$.  
\end{proposition}
\begin{proof}
First note that if $V$ is a vector field corresponding to a special form $\omega$, then
$L_V \omega = d(\iota_V\omega) + \iota_V d\omega = \iota_V d\omega$ because $\iota_V\omega=1$.
Therefore, $V$ is the characteristic vector field of $\omega$ if and only if $\omega(V)=1$ and $\iota_V d\omega = 0$.

Let $\omega$  be a special form on a neighborhood $U$ of $p$, and $V$ be the corresponding characteristic vector field. Let us take a form $\widetilde{\omega} = e^\varphi \omega$ and find conditions on $\varphi$ for $\widetilde{\omega}$ to be a special form. 

First, $\widetilde\omega(\widetilde V) = 1$ if and only if $\widetilde V = e^{-\varphi} V + W$, where $W \in \mathfrak{X}(\Delta)$, because  $\omega(V)=1$. 
Further, we have
\begin{equation}
d\widetilde\omega = e^\varphi d\varphi \wedge \omega + e^\varphi d\omega = 
d\varphi\wedge\widetilde\omega + e^\varphi d\omega.
\label{eq:2_1}
\end{equation}
Then
\begin{equation}
\iota_{\widetilde V} d\widetilde\omega = 
\iota_{\widetilde V} d\varphi\, \omega - d\varphi\, \iota_{\widetilde V} \omega + 
e^\varphi\iota_{\widetilde V} d\omega = 
\widetilde V\varphi\, \omega - d\varphi\, \omega(\widetilde V) + 
\iota_{V} d\omega + e^\varphi\iota_{W} d\omega; 
\label{eq:2_2}
\end{equation}
As $V$ is the characteristic vector field of $\omega$, we have $\iota_V d\omega = 0$, hence follows
\begin{equation}
\iota_{\widetilde{V}} d\widetilde{\omega} = \widetilde V\varphi \widetilde\omega - d\varphi + e^\varphi \iota_W d\omega.
\label{eq:2_3}
\end{equation}

Assume that $p$ does not lie in $\Sigma$, then, by Proposition~\ref{prop:2_2}, d), we have that $d\omega|_\Delta$ is nondegenerate form, so one can find a unique $W$ such that $d\varphi(W') = e^\varphi i_W d\omega (W')$.  By \eqref{eq:2_3}, with this $W$, $i_{\widetilde{V}} d \widetilde{\omega} = 0$ on $\Delta$. Also, if we substitute $V$ to the right hand side of \eqref{eq:2_3}, then we get $\widetilde{V}\varphi\,e^\varphi - V\varphi = 0$, 
since $\iota_W d\omega (V) = 2 d\omega(W,V) = -\iota_V d\omega(W) = 0$. Thus, we have found $W$ such that the corresponding vector field $\widetilde{V} = e^{-\varphi} + W$ is the characteristic vector field for $\widetilde \omega$ since $\widetilde\omega(\widetilde V) = 1$ and $\iota_{\widetilde{V}} d\widetilde{\omega} =0$. Thus, the form $\widetilde\omega$ is special, and we have proved a).

For $p \in \Sigma$, $d\omega |_\Delta = \lambda_\omega \Omega$, where $\Omega$ is the area form of the metric on $\Delta$. Thus, if $\omega$ and $\widetilde\omega$ are special, by \eqref{eq:2_3} we have that $d\varphi|_\Delta = -\frac{1}{2}\lambda_\omega e^\varphi \Omega(W, \cdot)$, therefore in this case we have that $ d\varphi|_\Delta = \lambda_\omega \xi$, where $\xi(W') = -\frac{1}{2} e^\varphi \Omega(W, W')$ is a nonzero $1$-form on $\Delta$.
Now, if $d\varphi|_\Delta = \lambda_\omega \xi$, then one can find $W$ such that $\xi(W') =  -\frac{1}{2} e^\varphi \Omega(W, W')$ for any $W' \in \mathfrak{X}(\Delta)$. If we now set $\widetilde V = e^{-\varphi} V +W$, then $\iota_{\widetilde{V}} d\widetilde{\omega} (W') = 0$, for any $W' \in \mathfrak{X}(\Delta)$. Also, as before, we have $\iota_{\widetilde{V}} d\widetilde{\omega} (V) = 0$, hence follows $\iota_{\widetilde{V}} d\widetilde{\omega} = 0$. Thus, $\widetilde\omega$ is special and we have proved b).
\end{proof}
\begin{remark}
It is clear that any symmetry of the nonholonomic surface maps a special form to a special form.
\end{remark}
 
\subsection{Adapted frame}
Assume that the distribution $\Delta$ is given by a special form $\omega$.
In a neighborhood $U$ of $p \in M$ take a frame field constructed in the following way.
Since $\lambda_\omega$ vanishes at $\Sigma$ and $d\lambda_\omega \ne 0$ at $\Sigma$, $U$ is foliated by the level surfaces $\Sigma_c = \lambda_\omega^{-1}(c) \cap U$, $c \in (-\alpha,\beta)$, $\alpha,\beta>0$, of $\lambda_\omega$.
Moreover, since $\Delta$ is transversal to $\Sigma=\Sigma_0$, then $\Delta$ is transversal to $\Sigma_c$, too. 
We take $E_1$ be the unit vector field of the line distribution $\Delta \cap T\Sigma_c$ on $U$, $c \in (-\alpha,\beta)$, (in fact, there are two such vector fields, they are opposite each other, we take one of them). 
The vector field $E_2 \in \mathfrak{X}(\Delta)$ is such that $E_1, E_2$ is positively oriented orthonormal frame field of $\Delta$. 
For $E_3$ we take the characteristic vector field of $\omega$.  

Now, in the structure equations $[E_i,E_j] = c^k_{ij} E_k$, we have $c^3_{31} = c^3_{23}=0$ because the flow of $E_3$ maps $\Delta$ to $\Delta$, and $c^3_{12} = \lambda_\omega$ by definition of $\lambda_\omega$. 

Let $\{\eta^1,\eta^2,\eta^3\}$ be the coframe dual to $\{E_a\}$. Note that $\eta^3 = \omega$.Then the structure equations are  
\begin{eqnarray}
&&
d\eta^1 = C^1_{23} \eta^2 \wedge \eta^3 + C^1_{31} \eta^3 \wedge \eta^1 + C^1_{12} \eta^1 \wedge \eta^2, 
\\
&&
d\eta^2 = C^2_{23} \eta^2 \wedge \eta^3 + C^2_{31} \eta^3 \wedge \eta^1 + C^2_{12} \eta^1 \wedge \eta^2, 
\\
&&
d\eta^3 = - \lambda_\omega  \eta^1 \wedge \eta^2, 
\label{eq:3_4_3}
\end{eqnarray}

From the coframe construction it follows that 
\begin{equation}
d \lambda_\omega = \lambda_2 \eta^2 + \lambda_3 \eta^3,
\label{eq:3_5}
\end{equation}
as $\lambda_1 = E_1 \lambda_\omega = 0$.

Applying the exterior differential to  \eqref{eq:3_4_3}, we get that  
\begin{equation*}
0 = - d\lambda_\omega \eta^1 \wedge \eta^2 + \lambda_\omega d\eta^1 \wedge \eta^2 - \lambda_\omega \eta^1 \wedge d\eta^2 = (-\lambda_3 + \lambda_\omega(C^1_{31}-C^2_{23}))\eta^1\wedge\eta^2\wedge\eta^3. 
\label{eq:3_5_1}
\end{equation*}
hence follows that
\begin{equation}
\lambda_3 = \lambda_\omega(C^1_{31}-C^2_{23}).
\label{eq:3_5_2}
\end{equation}
Therefore $\lambda_3 = 0$ on $\Sigma$, and, as $d\lambda_\omega \ne 0$ on $\Sigma$ we have that $\lambda_2 = E_2\lambda \ne 0$ on $\Sigma$.

\subsection{Change of the coframe}
It is clear that the frame $\{E_a\}$, and so the coframe $\{\eta^a\}$, is uniquely determined by the special form $\omega$. Now let us find how the coframe and the structure equations transform under a change of the special form $\omega$.

Let us take two special forms $\omega$ and $\widetilde\omega = e^\varphi \omega$, where $d\varphi |_\Delta = \lambda_\omega \xi$ (see Proposition~\ref{prop:2_3}, b)). Let $\{\eta^a\}$ and $\{\tilde\eta^a\}$ be the corresponding coframe fields. Then from construction we have:
\begin{eqnarray}
&&
\eta^1 = \cos\alpha\, \tilde\eta^1 + \sin\alpha\,  \tilde\eta^2 + a\,  \tilde\eta^3
\notag
\\
&&
\eta^2 = -\sin\alpha\,  \tilde\eta^1 + \cos\alpha\,  \tilde\eta^2 + b\,  \tilde\eta^3
\\
&&
\eta^3 = e^{-\varphi} \, \tilde\eta^3. 
\notag
\label{eq:3_6}
\end{eqnarray}

From  Proposition~\ref{prop:2_3}, b) it follows that 
$d\varphi = \lambda_\omega \xi_1 \eta^1 + \lambda_\omega \xi_2 \eta^2  + \varphi_3 \eta^3$,
Then
\begin{equation}
\begin{split}
d\, \tilde\eta^3 = d (e^\varphi \eta^3) = 
e^\varphi d\varphi \wedge \eta^3 + e^\varphi d\eta^3 = e^\varphi \lambda_\omega (\xi_1 \eta^1 + \xi_2 \eta^2) \wedge \eta^3 + e^\varphi (-\lambda_\omega \eta^1 \wedge \eta^2)
\\
=\lambda_\omega [(\xi_1 \eta^1 + \xi_2 \eta^2) \wedge \tilde\eta^3 - e^\varphi \eta^1 \wedge \eta^2].
\label{eq:3_7}
\end{split}
\end{equation}
From \eqref{eq:3_6} it follows that
\begin{equation*}
\xi_1 \eta^1 + \xi_2 \eta^2 = (\xi_1 \cos\alpha - \xi_2 \sin\alpha)\tilde\eta^1 + (\xi_1 \sin\alpha + \xi_2 \cos\alpha)\tilde\eta^2 + (a\xi_1 + b\xi_2) \tilde\eta^3.
\label{eq:3_8}
\end{equation*}
then 
\begin{equation}
(\xi_1 \eta^1 + \xi_2 \eta^2)\wedge \tilde\eta^3 = (\xi_1 \cos\alpha - \xi_2 \sin\alpha)\tilde \eta^1\wedge \tilde\eta^3 + (\xi_1 \sin\alpha + \xi_2 \cos\alpha)\tilde \eta^2\wedge \tilde\eta^3. 
\label{eq:3_9}
\end{equation}
Also 
\begin{equation}
\eta^1 \wedge \eta^2 = \tilde\eta^1 \wedge \tilde\eta^2 + (b\cos\alpha  + a\sin\alpha )\tilde\eta^1 \wedge \tilde\eta^3 + (b\sin\alpha  - a\cos\alpha )\tilde\eta^2 \wedge \tilde\eta^3. 
\label{eq:3_10}
\end{equation}
The equation \eqref{eq:3_4_3} written for the coframe $\{\tilde\eta^a\}$ gives 
$d\tilde\eta^3 = -\lambda_{\tilde{\omega}} \tilde\eta^1 \wedge \tilde\eta^2$, and from Proposition~\ref{prop:2_2} we have $\lambda_{\tilde{\omega}} = e^\varphi \lambda_\omega$.
Hence follows $d\tilde\eta^3 = - e^\varphi\lambda_{\omega} \tilde\eta^1 \wedge \tilde\eta^2$.
Thus \eqref{eq:3_7}, \eqref{eq:3_9}, and \eqref{eq:3_10} together give 
the equation system 
\begin{equation}
\left\{
\begin{array}{l}
\lambda_\omega[\xi_1 \cos\alpha - \xi_2\sin\alpha - e^\varphi(b\cos\alpha  + a\sin\alpha)]=0
\\
\lambda_\omega[\xi_1 \sin\alpha - \xi_2\cos\alpha - e^\varphi(b\sin\alpha  - a\cos\alpha)]=0
\end{array}
\right.
\label{eq:3_11}
\end{equation}
As $\lambda_\omega \ne 0$ almost everythere in $U$, we have that the expressions in brackets in \eqref{eq:3_11} vanish. Therefore, we arrive at the system
\begin{equation}
\left\{
\begin{array}{l}
(\xi_1 - e^\varphi b) \cos\alpha - (\xi_2+e^\varphi a)\sin\alpha = 0
\\
(\xi_1 - e^\varphi b) \sin\alpha - (\xi_2+e^\varphi a)\cos\alpha = 0
\end{array}
\right.
\label{eq:3_12}
\end{equation}
Thus we have found 
\begin{equation}
a = - e^{-\varphi} \xi_2, \qquad b =  e^{-\varphi} \xi_1.
\label{eq:3_13}
\end{equation}

Now it remains to find the function $\alpha$ in \eqref{eq:3_6}. To this end we use \eqref{eq:3_5}.
We have 
\begin{equation}
\begin{split}
d \lambda_{\widetilde\omega} = d (e^\varphi \lambda_\omega) = e^\varphi \lambda_\omega d\varphi + e^\varphi d\lambda_\omega = 
e^\varphi\lambda_\omega(\lambda_\omega \xi_1 \eta^1+\lambda_\omega \xi_2 \eta^2+\varphi_3\eta^3) + e^\varphi(\lambda_2 \eta^2 + \lambda_3 \eta^3) =
\\
e^\varphi [ \lambda_\omega^2 \xi_1 \eta^1 + (\lambda_\omega^2 \xi_2 + \lambda_2) \eta^2 + (\lambda_\omega \varphi_3 + \lambda_3)\eta^3].
\end{split}
\label{eq:3_14}
\end{equation}
To \eqref{eq:3_14} we substitute \eqref{eq:3_6} and get that 
\begin{equation}
\begin{split}
d\lambda_{\widetilde\omega} = e^\varphi [\lambda_\omega^2 \xi_1 \cos\alpha - (\lambda_\omega^2\xi_2 + \lambda_2) \sin\alpha] \tilde\eta^1 + 
e^\varphi[\lambda_\omega^2 \xi_1 \sin\alpha + (\lambda_\omega^2\xi_2 + \lambda_2) \cos\alpha] \tilde\eta^2 +
\\
 (\lambda_\omega \varphi_3 + \lambda_2 \xi_1 + \lambda_3)\tilde\eta^3.
\end{split}
\label{eq:3_15}
\end{equation}
From this follows that 
\begin{equation}
d\lambda_{\widetilde\omega} = \widetilde\lambda_1 \tilde\eta^1 + \widetilde\lambda_2 \tilde\eta^2 + \widetilde\lambda_3 \tilde\eta^3,  
\label{eq:3_16}
\end{equation}
where
\begin{eqnarray}
&&
\widetilde\lambda_1 = e^\varphi [\lambda_\omega^2 \xi_1 \cos\alpha - (\lambda_\omega^2\xi_2 + \lambda_2) \sin\alpha],
\label{eq:3_17_1}
\\
&&
\widetilde\lambda_2 = e^\varphi[\lambda_\omega^2 \xi_1 \sin\alpha + (\lambda_\omega^2\xi_2 + \lambda_2) \cos\alpha] 
\label{eq:3_17_2}
\\
&&
\widetilde\lambda_3 = \lambda_\omega \varphi_3 + \lambda_2 \xi_1 + \lambda_3
\label{eq:3_17_3}
\end{eqnarray}
From \eqref{eq:3_5_2} it follows that $\lambda_3 = \lambda_\omega f$, in the same way, $\widetilde\lambda_3 = \lambda_{\widetilde\omega} \tilde f$, where $f$, $\tilde f$ are functions. Then $\widetilde\lambda_3 = e^\varphi \lambda_\omega \tilde f$. 
As $\lambda_2 \ne 0$ at $\Sigma$ (see reasoning below \eqref{eq:3_5_2}), from \eqref{eq:3_17_3} we have that  
\begin{equation}
\xi_1 = \lambda_\omega \mu.  
\label{eq:3_18}
\end{equation}
Also, by \eqref{eq:3_5}, $\widetilde\lambda_1 = 0$, hence \eqref{eq:3_17_1} and \eqref{eq:3_18} give us the expression for $\alpha$:
\begin{equation}
\tan\alpha = \frac{\lambda_\omega^3 \mu}{\lambda_\omega^2\xi_2+\lambda_2}. 
\label{eq:3_19}
\end{equation}
Thus we have proved 
\begin{proposition}
Let $\left\{ \eta^a\right\}$ be the adapted frame determined by a special form $\omega$, and $\left\{ \tilde\eta^a\right\}$ be the adapted frame determined by a special form $\widetilde\omega = e^\varphi\omega$. Then 
\begin{eqnarray}
&&
\eta^1 = \cos\alpha\, \tilde\eta^1 + \sin\alpha\,  \tilde\eta^2 - e^{-\varphi}\xi_2 \,  \tilde\eta^3
\label{eq:3_20_1}
\\
&&
\eta^2 = -\sin\alpha\,  \tilde\eta^1 + \cos\alpha\,  \tilde\eta^2 + e^{-\varphi}\xi_1\,  \tilde\eta^3
\label{eq:3_20_2}
\\
&&
\eta^3 = e^{-\varphi} \, \tilde\eta^3, 
\label{eq:3_20_3}
\end{eqnarray}
Here functions $\xi_1$, $\xi_2$, and $\alpha$ are determined by $\varphi$ in the following way: 
\begin{eqnarray}
&& 
E_1\varphi = \varphi_1 = \lambda_\omega \xi_1 = \lambda_\omega^2\mu
\label{eq:3_21_1}
\\
&&
E_2\varphi = \varphi_2 = \lambda_\omega \xi_2
\label{eq:3_21_2}
\\
&&
\tan\alpha = \frac{\lambda_\omega^3 \mu}{\lambda_\omega^2\xi_2+\lambda_2}, 
\label{eq:3_21_3}
\end{eqnarray}
where $\left\{ E_a \right\}$ is the adapted frame dual to $\left\{ \eta^a \right\}$, $\mu$ is a function, and $\lambda_2 = E_2\lambda_\omega$.
\label{prop:2_5}
\end{proposition}

\subsection{Invarians of sub-Riemannian surface along the singular surface}

Let $\left\{ \eta^a\right\}$ be the adapted frame determined by a special form $\omega$, and $\left\{ \tilde\eta^a\right\}$ be the adapted frame determined by a special form $\widetilde\omega = e^\varphi\omega$. Then we have the structure equations: $d\eta^a = C^a_{bc} \eta^b \wedge \eta^c$ and  $d\tilde\eta^a = \tilde C^a_{bc} \tilde\eta^b \wedge \tilde\eta^c$. Let us denote the restrictions of the structure functions $C^a_{bc}$ and $\tilde C^a_{bc}$ to the surface $\Sigma$ by  $Q^a_{bc}$ and $\tilde Q^a_{bc}$, respectively.
Let us find relation between  $Q^a_{bc}$ and $\tilde Q^a_{bc}$, 

The surface $\Sigma$ is given by the equation $\lambda_\omega = 0$. Using Proposition~\ref{prop:2_5}, we get that at the points of $\Sigma$ the equalities \eqref{eq:3_20_1}--\eqref{eq:3_20_3} are written as follows:
\begin{eqnarray}
&&
\eta^1 = \tilde\eta^1 - e^{-\varphi}\xi_2 \,  \tilde\eta^3
\label{eq:3_22_1}
\\
&&
\eta^2 = \tilde\eta^2  
\label{eq:3_22_2}
\\
&&
\eta^3 = e^{-\varphi} \, \tilde\eta^3, 
\label{eq:3_22_3}
\end{eqnarray}
Now, let us take the exterior derivative of \eqref{eq:3_20_1}  
\begin{equation}
\begin{split}
d\eta^1 = -\sin\alpha d\alpha\wedge\tilde\eta^1 + \cos\alpha d\tilde\eta^1 + \cos\alpha  d\alpha\wedge\tilde\eta^2 + \sin\alpha d\tilde\eta^2 + 
\\
 e^{-\varphi} \xi_2 d\varphi\wedge\tilde\eta^3 - e^{-\varphi} d\xi_2\wedge\tilde\eta^3 - e^{-\varphi} \xi_2 d\tilde\eta^3. 
\end{split}
\label{eq:3_23_1}
\end{equation}
and take the result at a point of $\Sigma$, then we have, by \eqref{eq:3_21_3}, that $\cos\alpha=1$, $\sin\alpha=0$, $d\alpha = 0$. Also, by \eqref{eq:3_21_1} and  \eqref{eq:3_21_2}, 
\begin{equation}
d\varphi \wedge \tilde\eta^3 = (\varphi_1 \eta^1 + \varphi_2 \eta^2 + \varphi_3 \eta^3)\wedge\tilde\eta^3 = \lambda_\omega \xi_1 \eta_1 \wedge \tilde\eta^3 + \lambda_\omega \xi_2 \eta_2 \wedge \tilde\eta^3,
\label{eq:3_24}
\end{equation}
hence, at points of $\Sigma$, $d\varphi \wedge \tilde\eta^3 = 0$. In addition, from \eqref{eq:3_4_3} it follows that at $\Sigma$, $d\tilde\eta^3=0$. Therefore, on $\Sigma$ we have 
\begin{equation}
d\eta^1 = d\tilde\eta^1 - e^{-\varphi} d\xi_2 \wedge \tilde\eta^3 = d\tilde\eta^1 - d\xi_2 \wedge \eta^3. 
\label{eq:3_25}
\end{equation}
From \eqref{eq:3_22_1}--\eqref{eq:3_22_3} it follows that on $\Sigma$ we have 
\begin{eqnarray}
&&
\tilde\eta^1 = \eta^1 + \xi_2 \,\eta^3
\label{eq:3_26_1}
\\
&&
\tilde\eta^2 = \eta^2  
\label{eq:3_26_2}
\\
&&
\tilde\eta^3 = e^{\varphi} \,\eta^3, 
\label{eq:3_26_3}
\end{eqnarray}
Then, at points in $\Sigma$ we have 
\begin{multline}
d\tilde\eta^1 = \tilde Q^1_{23} \tilde\eta^2 \wedge \tilde\eta^3 + \tilde Q^1_{31} \tilde\eta^3 \wedge \tilde\eta^1 + \tilde Q^1_{12} \tilde\eta^1 \wedge \tilde\eta^2 = 
\\
(e^\varphi \tilde Q^1_{23} - \xi_2 \tilde Q^1_{12}) \eta^2 \wedge \eta^3 +
e^\varphi \tilde Q^1_{31} \eta^3 \wedge \eta^1 +
\tilde Q^1_{12} \eta^1 \wedge \eta^2.
\label{eq:3_27}
\end{multline}
Set $d\xi_2 = \xi_{21}\eta^1 + \xi_{22}\eta^2 + \xi_{23}\eta^3$, and substitute it together with \eqref{eq:3_27} to \eqref{eq:3_25}.   
Then we get 
\begin{equation}
d\eta^1 = 
(e^\varphi \tilde Q^1_{23} - \xi_2 \tilde Q^1_{12} - \xi_{22}) \eta^2 \wedge \eta^3 +
(e^\varphi \tilde Q^1_{31} + \xi_{21}) \eta^3 \wedge \eta^1 +
\tilde Q^1_{12} \eta^1 \wedge \eta^2.  
\label{eq:3_28}
\end{equation}
Thus,
\begin{equation}
Q^1_{23} = e^\varphi \tilde Q^1_{23} - \xi_2 \tilde Q^1_{12} - \xi_{22}; 
\quad
Q^1_{31} =  e^\varphi \tilde Q^1_{31} + \xi_{21};
\quad
Q^1_{12} = \tilde Q^1_{12}. 
\label{eq:3_29}
\end{equation}
In the same manner we prove that $d\eta^2 = d\tilde\eta^2 + d\xi_1\wedge\eta^3$, By \eqref{eq:3_18}, we have 
\begin{equation}
d\xi_1 = \mu  d\lambda_\omega + \lambda_\omega d\mu = \mu(\lambda_2 \eta^2 + \lambda_3\eta^3) + \lambda_\omega d\mu
\label{eq:3_30}
\end{equation}  
By \eqref{eq:3_5_2}, we get that $\lambda_3 = 0$ on $\Sigma$, so at points of $\Sigma$ we have
\begin{equation}
d\xi_1 = \mu\lambda_2 \eta^2.
\label{eq:3_31}
\end{equation}
Then we get 
\begin{equation}
Q^2_{23} = e^\varphi \tilde Q^2_{23} - \xi_2 \tilde Q^2_{12} + \mu\lambda_2; 
\quad
Q^2_{31} =  e^\varphi \tilde Q^2_{31};
\quad
Q^2_{12} = \tilde Q^2_{12}. 
\end{equation}
Thus we have proved
\begin{proposition}
Let $d\eta^a = C^a_{bc} \eta^b \wedge \eta^c$ be the structure equations of the adapted frame of the sub-Riemmanian surface in a neighborhood of a point in $\Sigma$. 
The functions $Q^1_{12} = C^1_{12}|_\Sigma$ and $Q^2_{12} = C^2_{12}|_\Sigma$ do not depend on the choice of adapted frame, and so these functions are invariants of the surface.
\end{proposition}

\begin{corollary}
If $V$ is an infinitesimal symmetry of the sub-Riemannian surface, then $V Q^1_{12}= 0$, and $V Q^2_{12} = 0$. 
\end{corollary}
\begin{proof}
Any symmetry $f$ of the sub-Riemannian surface maps $\Sigma$ onto itself, and sends a special form to a special form and the corresponding adapted frame to the corresponding adapted frame. Therefore, $f^* Q^1_{12} = Q^1_{12}$ and  $f^* Q^2_{12} = Q^2_{12}$. From this follows that  an  infinitesimal symmetry $V$ is tangent to $\Sigma$ and $V Q^1_{12} = V Q^2_{12} = 0$.  
\end{proof}
%%%%%%%%%%%%%%%%%%%%%%%%%%%%%%%%%%%

%\end{linenumbers} %%%%%%%%%% ERASE WHEN YOU HAD FINISHED %%%%%%%%%%%%


\begin{thebibliography}{99}


\bibitem {Aminov} Yu. Aminov, The geometry of vector fields, Gordon and Breach Publishers,  Amsterdam, 2000. % Yes, in russian 1990

\bibitem{Bieliavsky} P. Bieliavsky, E. Falbel, C. Gorodski, The classification of simple-connected contact sub-Riemannian symmetric spaces,  Pacific J. of Math. Vol. 188, 1999 % Yes

\bibitem{Bloch1} A.M. Bloch, J.E. Marsden, D. Zenkov, Nonholonomic Dynamics, Notices of AMS, Vol 52,3, 2005. % Yes

\bibitem{Bloch2} A.M. Bloch, P.S. Krishnaprasad, J.E. Marsden, R. Murray, Nonholonomic Mechanical Systems with Symmetry, Arch. Rational Mech. Anal. 169, 1996. % Yes

\bibitem {Ehlers1} K. Ehlers, Geometric Equivalence of Nonholonomic three-manifolds, Proceedinds of the fourth International Conference on Dynamical System and Differential Equations, May 24 – 27, (2002), Wilmington, NC, USA pp. 246–255. % Yes

\bibitem{Ehlers2} K. Ehlers, J. Koiller, P. M. Rios, Nonholonomic Systems:  Cartan's equivalence
and Hamiltonization, Vienna, Preprint ESI 1389 (2003) % Yes

\bibitem{Falbel1} E. Falbel, C. Gorodski, Sub-Riemannian Homegeneous Spaces in Dimensions 3 and 4, Geometriae Dedicadta 62, 1996. % Yes

\bibitem{Falbel2} E. Falbel, C. Gorodski, On contact sub-Riemannian symmetries spaces, Annales scientifiques de l'É.N.S. 4e série, tome 28, nº5, 1995, p.571-589 % Yes

\bibitem{Fels} M. Fels, P.J. Olver, Moving Coframes: I. A Practical Algorithm, Acta Applicandae Mathematicae 51, Kluwer Academic Publishers. Printed in the Netherlands, 1998, p.161–213. % Yes

\bibitem{Gardner} R.B. Gardner, The Method of Equivalence and Its Applications, SIAM, Philadelphia, 1989. % Yes

\bibitem{Grozman} P.Ya. Grozman, D.A. Leites, Nonholonomic Riemann and Weyl tensor for flag manifolds, Theoretical and Mathematical Physics, 153, 2007. % Yes

\bibitem{Hughen} K. Hughen, The Geometry of Subriemannian Three-Manifolds, PhD. Thesis, Duke Univ. 1995. % Yes

\bibitem{KN} S. Kobayashi, K. Nomizu, Foundations of Differential Geometry, Vols. I and II, Interscience, London, 1963. % Yes, ... search in basic books

\bibitem{Montgomery} R. Montgomery, A tour of Subriemannian Geometries, Their Geodesics and Applications,  Mathematical Surveys and Monographs, V. 91, AMS, Providence, 2002. % Yes, in two parts

\bibitem{Olver}	P.J. Olver, Equivalence, invariants, and symmetry, Cambridge University Press, Cambridge, 1995.  % Yes

\bibitem{Pavlov} V.P. Pavlov, V.M. Sergeev, Thermodynamics from the differential geometry standpoint, Theoretical and Mathematical Physics, 157, 2008. % Yes, in russian and english

\bibitem{Pereira} K.V. Pereira, Vector fields, invariant varieties and linear system, Annales de L'Institute Fourier Tome 51, nº5, 2001. % Yes

\bibitem{Schempp} W. Schempp, Sub-Riemannian Geometry and Clinical Magnetic Resonance Tomography, Math. Meth.  Appl. Sci., 22, 1999. % Yes 

\bibitem{Shapukov} B.N. Shapukov, Connections on dlfferentiable bundles, Itog. Nauki Tekh. Probl. Geom.,
15, 61-93, 1983. % Yes

\bibitem{Strichartz} R.S. Strichartz, Sub-Riemannian Geometry, J. Differential Geometry, 24, 1986. % Yes

\bibitem{Sachkov} Y. L. Sachkov. Symmetries of flat rank two distributions and sub-riemannian structures, Trans. AMS, V. 356(2), 2003, 457--494.


\bibitem{Vershik} A.M. Vershik, V. Gerhskovich, Nonholonomic Dynamical Systems, Geometry of Distributions and Variational Problems, Dynamical Systems VII. Ed. V.I. Arnol’d and S.P. Novikov, vol. 16 of the Encyclopedia of Mathematical Sciences series, Springer-Verlag, NY, 1994 % No

\bibitem{Zhitomirskii} M. Zhitomirskii, Typical singularities of Differential 1-forms and Pfaffian equations, Translation of Mathematical Monographs, Vol. 113, AMS, 1992 % No

\end{thebibliography}
\end{document}